\documentclass[10pt,a4paper]{article}

\usepackage[latin1]{inputenc}
\usepackage{amsmath}
\usepackage{amsfonts}
\usepackage{amssymb}
\usepackage{amsthm}
\usepackage{graphics}
\usepackage{graphicx}
\usepackage[affil-it]{authblk}
\usepackage{tikz}
\usepackage{comment}
\usepackage{fixmath}

\usepackage{enumerate}

\usepackage{subfig}

\usepackage{adjustbox}
\newcommand*\circled[1]{\tikz[baseline=(char.base)]{
            \node[shape=circle,draw,inner sep=2pt] (char) {#1};}}
\renewcommand{\geq}{\geqslant}
\renewcommand{\leq}{\leqslant}

\newcommand{\sB}{{\mathcal B}}
\newcommand{\sC}{{\mathcal C}}
\newcommand{\sP}{{\mathcal P}}

\newtheorem{theorem}{Theorem}[section]
\newtheorem{proposition}[theorem]{Proposition}
\newtheorem{corollary}[theorem]{Corollary}

\date{today}

\def\Ggraph{
\draw [black,fill] (-1,1) circle [radius=0.1] node [above] {2};
\draw [black,fill] (1,-1) circle [radius=0.1]node [below] {7};
\draw [black,fill] (0,3) circle [radius=0.1]node [left] {3};
\draw [black,fill] (4,1) circle [radius=0.1]node [above] {5};
\draw [black,fill] (-3,2) circle [radius=0.1]node [above] {1};
\draw [black,fill] (4,0) circle [radius=0.1]node [below] {6};
\draw [black,fill] (2,4) circle [radius=0.1]node [above] {4};

\draw [thin] (-1,1) to (4,1); \draw [black] (2,1.2) node {$d$};
\draw [thin] (-1,1) to (-3,2); \draw [black] (-2,1.7) node {$a$};
\draw [thin] (-1,1) to (0,3); \draw [black] (-0.6,2.2) node {$b$};
\draw [thin] (-1,1) to (2,4); \draw [black] (0.8,3) node {$c$};
\draw [thin] (2,4) to (4,1); \draw [black] (3.4,2.3) node {$i$};
\draw [thin] (2,4) to (0,3); \draw [black] (1,3.7) node {$f$};
\draw [thin] (0,3) to (1,-1); \draw [black] (0.7,-0.5) node {$g$};
\draw [thin] (1,-1) to (4,0); \draw [black] (2.5,-0.3) node {$h$};
\draw [thin] (4,0) to (4,1); \draw [black] (4.2,0.5) node {$k$};
\draw [thin] (4,0) to (-1,1); \draw [black] (2,0.5) node {$e$};
\draw [thin] (2,4) to (4,0); \draw [black] (2.4,2.8) node {$j$};
}

\def\Gtwograph{
\Ggraph
\draw [thin, white, dashed] (0,3) to (1,-1);
\draw [thin, white, dashed] (1,-1) to (4,0);
\draw [white, fill] (1,-1) circle [radius=0.09];
\draw [black] (1,-1) circle [radius=0.1];
}

\def\Gthreegraph{
\Gtwograph
\draw [thin, white, dashed] (-1,1) to (0,3);
\draw [thin, white, dashed] (2,4) to (0,3);
\draw [thin, white, dashed] (0,3) to (1,-1);
\draw [white, fill] (0,3) circle [radius=0.09];
\draw [black] (0,3) circle [radius=0.1];
}
\def\Gfourgraph{
\Gthreegraph
}
\def\Gfivegraph{
\Gfourgraph
\draw [thin, white, dashed] (2,4) to (4,1);
\draw [thin, white, dashed] (2,4) to (4,0);
\draw [thin, white, dashed] (-1,1) to (2,4);
\draw [white, fill] (2,4) circle [radius=0.09];
}
\def\Gsixgraph{
\Gfivegraph
}
\def\Gsevengraph{
\Gsixgraph
\draw [thin, white, dashed] (4,0) to (4,1);
\draw [thin, white, dashed] (-1,1) to (4,1);
\draw [white, fill] (4,1) circle [radius=0.09];
}
\def\Geightgraph{
\Gsevengraph
\draw [thin, white, dashed] (-1,1) to (4,1);
}
\def\Gninegraph{
\Geightgraph
\draw [thin, white, dashed] (-1,1) to (4,0);
\draw [white, fill] (4,0) circle [radius=0.09];
}
\def\Gtengraph{
\Gninegraph
\draw [thin, white, dashed] (-3,2) to (-1,1);
\draw [white, fill] (-1,1) circle [radius=0.09];
}
\def\Gelevengraph{
\Gtengraph
\draw [white, fill] (-3,2) circle [radius=0.09];
}

\def\LGgraph{
\draw [black,fill] (0.5,1) circle [radius=0.1]; \draw [black] (0.75,1.1) node {$c$};
\draw [black,fill] (1,2) circle [radius=0.1]; \draw [black] (1.25,2) node {$d$};
\draw [black,fill] (-1,2) circle [radius=0.1]; \draw [black] (-1.3,2) node {$a$};
\draw [black,fill] (-0.5,1) circle [radius=0.1]; \draw [black] (-0.8,1) node {$b$};
\draw [black,fill] (0,3) circle [radius=0.1]; \draw [black] (0,3.25) node {$e$};

\draw [thin] (0.5,1) to (1,2);
\draw [thin] (0.5,1) to (-1,2);
\draw [thin] (0.5,1) to (-0.5,1);
\draw [thin] (0.5,1) to (0,3);
\draw [thin] (1,2) to (-1,2);
\draw [thin] (1,2) to (-0.5,1);
\draw [thin] (1,2) to (0,3);
\draw [thin] (-1,2) to (-0.5,1);
\draw [thin] (-1,2) to (0,3);
\draw [thin] (-0.5,1) to (0,3);
\draw [thin] (-0.5,1) to (1,-1);
\draw [thin] (-0.5,1) to (2,0);
\draw [thin] (0.5,1) to (1,-1);
\draw [thin] (0.5,1) to (4,1);
\draw [thin] (0.5,1) to (3,3);
\draw [thin] (1,2) to (4,1);
\draw [thin] (1,2) to (2,4);
\draw [thin] (0,3) to (4,0);
\draw [thin] (0,3) to (2,4);
\draw [thin] (0,3) to (3,3);
\draw [thin] (1,-1) to (4,1);
\draw [thin] (1,-1) to (2,0);
\draw [thin] (1,-1) to (3,3);
\draw [thin] (2,0) to (4,0);
\draw [thin] (4,0) to (3,3);
\draw [thin] (4,0) to (2,4);
\draw [thin] (4,1) to (3,3);
\draw [thin] (4,1) to (2,4);
\draw [thin] (3,3) to (2,4);

\draw [black,fill] (4,1) circle [radius=0.1]; \draw [black] (4.3,1) node {$i$};
\draw [black,fill] (1,-1) circle [radius=0.1];\draw [black] (0.8,-1.1) node {$f$};
\draw [black,fill] (4,0) circle [radius=0.1]; \draw [black] (4.3,0) node {$h$};
\draw [black,fill] (2,4) circle [radius=0.1]; \draw [black] (2,4.3) node {$k$};
\draw [black,fill] (2,0) circle [radius=0.1]; \draw [black] (2,-0.25) node {$g$};
\draw [black,fill] (3,3) circle [radius=0.1]; \draw [black] (3.3,3) node {$j$};
}

\begin{document}

\title{Brushing Number and Zero-Forcing Number of Graphs and their Line Graphs}

\author{Aras Erzurumluo\u{g}lu%
 \thanks{e-mail address: \texttt{aerzurumluog@mun.ca}}}

\affil{Department of Mathematics and Statistics\\
       Memorial University of Newfoundland\\
       St.~John's, NL, Canada A1C 5S7}

\author{Karen Meagher%
 \thanks{e-mail address: \texttt{karen.meagher@uregina.ca}}}

\affil{Department of Mathematics and Statistics\\
University of Regina\\
Regina, SK, Canada S4S 0A2}

\author{David A.~Pike%
 \thanks{e-mail address: \texttt{dapike@mun.ca}}}

\affil{Department of Mathematics and Statistics\\
       Memorial University of Newfoundland\\
       St.~John's, NL, Canada A1C 5S7}

\date{\today}

\maketitle

\makeatletter

\def\@maketitle{%
  \newpage
  \null
  \vskip 2em%
  \begin{center}%
  \let \footnote \thanks
    {\Large\bfseries \@title \par}%
    \vskip 1.5em%
    {\normalsize
      \lineskip .5em%
      \begin{tabular}[t]{c}%
        \@author
      \end{tabular}\par}%
    %\vskip 1em%
    %{\normalsize \@date}%
  \end{center}%
  \par
  \vskip 1.5em}
\makeatother

\begin{abstract}
  In this paper we compare the brushing number of a graph with the
  zero-forcing number of its line graph. We prove that the zero-forcing
  number of the line graph is an upper bound for the brushing number
  by constructing a brush configuration based on a zero-forcing set
  for the line graph. Using a similar construction, we also prove the
  conjecture that the zero-forcing number of a graph is no more than
  the zero-forcing number of its line graph; moreover we prove that the brushing number of a graph is no more than
  the brushing number of its line graph. All three bounds are shown to be tight.
\end{abstract}

\vspace*{\baselineskip}

\vspace*{\baselineskip}

\section{Introduction}

Recently there has been much research on different edge and node
search algorithms for graphs, typically based on different
applications and modelling of different situations.  Each variation
leads to new graph parameters and it is interesting to compare these
different parameters, particularly where the parameters have different
motivations.  For example, in~\cite{MR3440126} it is proven that the
zero-forcing number of a graph is equal to the fast-mixed search
number and in~\cite{MR2509446} a connection between the imbalance of a graph
and brushing is established.
%it is shown that the imbalance of a graph is twice the brush number $b(G)$.
In this paper we compare the zero-forcing number and the brushing number of a graph. All the graphs that we consider
are simple graphs, meaning no graph has a loop or multiple edges.

To introduce the zero-forcing number of a graph $G$, we begin with a colouring
of the vertices of $G$ with the colours black and white.
%In zero forcing for graphs we consider the vertices of the graph to be coloured either black or white.
A black vertex can \textsl{force} a
white vertex to change its colour to black according to a colour-changing rule:
if $v$ is black and $w$ is a white neighbour of $v$, then $v$ can force $w$ to become black only if $w$ is the only white vertex that is adjacent to $v$.
%The colour change rule for standard zero forcing is that a black vertex can force
%an adjacent white vertex to black if it is the only white vertex adjacent to that black vertex
%(there are other rules that define other types of zero forcing).
A set of vertices in a graph is a \textsl{zero-forcing set} for the graph if, when the vertices in this set are
initially set to black and the colour-changing rule is applied
repeatedly, all the vertices of the graph are eventually forced to
black. The \textsl{zero-forcing number} of a graph $G$ is the size of the
smallest zero-forcing set for the graph and is denoted by $Z(G)$.
For additional background on zero forcing for graphs, see~\cite{MR2388646, MR2645093, MR3010007, HCY}.
%There are many references for zero forcing for graphs; we suggest the following~\cite{MR2388646, MR2645093, MR3010007, HCY}.

\iffalse
The term ``zero forcing'' is based on an algebraic property of a set
of vertices in a graph. One of the most interesting facts about zero
forcing is that sets with this algebraic property can be defined
strictly in terms of graph properties. This is how we will define
zero-forcing sets; details on the connections to algebraic properties
can be found in~\cite{MR2388646}.
\fi

The brushing number of a graph is motivated by a variant of graph searching.
%and influenced by chip firing in graphs.
The variation we
consider here was introduced in~\cite{mckeil} and explored in more
detail in~\cite{MR3176702} and~\cite{Penso2015}.
Specifically,
we start with a graph $G$ that models a situation of
contamination, meaning that initially every edge and every vertex
is deemed to be \textsl{dirty}.
Cleaning agents called \textsl{brushes} are placed at some of the vertices (this is
the \textsl{initial configuration} of brushes) and there is a process by
which the edges and vertices are subsequently \textsl{cleaned}.
Drawing terminology from the realm of chip firing,
a single vertex \textsl{fires} at each step in this process.
A vertex $v$ is permitted to fire only if the number of brushes at $v$ at the time that it fires
is at least the degree of $v$ (that is, the degree in the graph as it is at the time).
When a vertex $v$ fires, the brushes on $v$ clean $v$, and at least one unique brush
traverses each edge incident with $v$, thereby cleaning the dirty edges that were incident with $v$.
%As a brush moves along an edge, the brush cleans the edge.
At the end of the step each brush from $v$ is placed at the vertex adjacent to $v$ at the endpoint of the edge it traversed
(excess brushes are allowed to remain at $v$, although they then cease to have any future role).
The vertex $v$ and all of its incident edges (which are now
clean) are removed from the graph and the process continues (instead of removing them from the graph one can alternatively represent the clean edges by dashed lines, and the clean vertices as hollow circles). The
process is complete when there are no vertices that can fire.
Any edges or vertices that survive all remain dirty,
so if the remaining graph is empty then the initial configuration was capable of cleaning the original graph.
%If the graph that remains (these are all the edges and vertices which are
%still dirty) is empty, then the initial configuration can clean the graph.

The \textsl{brushing number} of a graph is the minimum number of
brushes needed for some initial configuration to clean the graph. For a
graph $G$, this is denoted by $B(G)$. There are several variants of
this parameter; here and in \cite{MR3176702} edges are allowed to be traversed by more than one
brush, and each edge is traversed during at most one step of the cleaning process. Alternatively we could
require that each edge is traversed by only one brush; the number of
brushes required in this scenario is denoted by $b(G)$ and was studied in~\cite{MR2476825} and~\cite{MNP2009}.
It is clear that a brushing strategy with this restriction is also a brushing strategy in our setting, so
for any graph $G$ it holds that $B(G) \leq b(G)$.
It was shown in~\cite{MR2509446} that $b(G)$ is equal to half of the
minimum total imbalance of the graph $G$, which in turn led to a proof
that shows that $b(G)$ is ${\mathcal{NP}}$-hard.

\iffalse
Computing the zero-forcing number for graphs is also an ${\mathcal{NP}}$-hard problem
(see Theorem 3.1 in~\cite{MR3440126} and Theorems 6.3, 6.5, Corollary 6.6 in~\cite{MR3110517}). The zero-forcing number and the brushing number
are not minor monotone~\cite{MR3010007}: the operations of
edge-deletion and edge-contraction, may either increase, decrease or
not change either of the zero-forcing number and the brushing number of a graph.
\fi

To demonstrate these definitions, we will give the value of both $Z(G)$ and
$B(G)$ for some well-known graphs.  As is usual, $K_n$ denotes the
complete graph on $n$ vertices, $C_n$ is the cycle on $n$ vertices,
$P_n$ is the path with $n$ vertices and $K_{m,n}$ is the complete
bipartite graph; in particular, $K_{1,n}$ is the star with $n$ edges.

\begin{proposition}\label{prop:easybounds} For $n \geq 3$
\begin{enumerate}[(i)]
\item $B(K_{n}) = \left\lfloor \frac{n^2}{4} \right\rfloor$ \textnormal{(see \cite{MR3176702})},
\item $Z(K_{n}) = n-1$,
\item $B(K_{1, n})=\lceil n/2 \rceil $ \textnormal{(see \cite{mckeil})},
\item $Z(K_{1, n}) = n-1$,
\item $B(P_{n})=Z(P_{n})=1$,
\item $B(C_{n})=Z(C_{n})=2$.
\end{enumerate}
\end{proposition}

The first two statements show that $B(K_{n}) > Z(K_{n})$ for $n \geq
4$. Moreover, by taking $n$ sufficiently large, $B(K_{n}) - Z(K_{n})$ can
be made arbitrarily large.  Conversely, the next two statements show
that $B(K_{1, n}) < Z(K_{1, n})$ for $n \geq 4$; and by taking $n$ sufficiently
large $Z(K_{1, n}) - B(K_{1, n})$ can be made arbitrarily
large. This confirms that when considering %we should not be comparing 
the brushing number and the zero-forcing number of the same graph we should not be trying to bound one by the other. In the brushing process the brushes traverse each edge, so rather than
comparing the brushing number of a graph to the zero-forcing number of
the graph, we compare the brushing number to the zero-forcing number
of the line graph. For a graph $G$ define the \textsl{line graph of
  $G$}, denoted by $L(G)$, to be the graph with a vertex for each edge of $G$ where two of
these vertices in $L(G)$ are adjacent if and only if the corresponding edges are incident in
$G$. %The line graph of $G$ is denoted by $L(G)$.
%Note that $G$ is connected exactly when $L(G)$ is connected.
If we label the
vertices of $G$ by $v_i$, then the vertices of $L(G)$ can be labelled
by the edges $\{v_i, v_j\}$ of $G$. If two distinct vertices of
$L(G)$, say $v= \{v_i, v_j\}$ and $w = \{w_i, w_j\}$, are adjacent in
$L(G)$, then $v \cap w$ is non-empty (it is exactly the vertex in $G$
that is on both edges).

The cycle $C_n$ is unusual in terms of its line graph, since it is the
only connected graph that is isomorphic to its own line graph. So
Proposition~\ref{prop:easybounds} implies that for $n\geq 3$
\[
B(C_{n}) = B(L(C_{n})) = 2 = Z(C_{n}) = Z(L(C_{n})).
\]

In this paper we explore how the parameters $B(G)$, $B(L(G))$, $Z(G)$ and $Z(L(G))$ are
related to one another more generally.
In Theorem~\ref{theorem1} we prove that $B(G) \leq Z(L(G))$, and in Corollary~\ref{corollary} it is further established
that $B(G) \leq b(G) \leq Z(L(G))$.  The example of the cycle shows that
these bounds cannot be improved (these bounds are also tight for the path $P_n$ on
$n \geq 2$ vertices).

The line graph of the complete graph $K_n$ is the Johnson graph
$J(n,2)$. In~\cite{CCgraphs}, it is shown that the zero-forcing number
of $J(n,2)$ is $\binom{n}{2}$, and thus for $n \geq 3$
\begin{align}\label{eq:complete}
B(K_{n}) = \left\lfloor \frac{n^2}{4} \right\rfloor < \binom{n}{2} = Z(L(K_n)).
\end{align}
So not only can the inequality $B(G) \leq Z(L(G))$ be strict, but the
difference can be arbitrarily large. Further, for $n \geq 4$
\begin{align}\label{eq:star}
 B(K_{1, n})=\lceil n/2 \rceil  <  n-1 = Z(K_n) = Z( L(K_{1,n}) )
\end{align}
which shows that even for trees the difference between the brushing
number and the zero-forcing number can be arbitrarily large.

%Further motivation to compare the brushing number of a graph to the zero-forcing number of its line graph comes from~\cite{eroh}.
In~\cite{eroh} it is proved that $Z(G) \leq
2Z(L(G))$ for any non-trivial graph $G$.  Moreover, it is proved that $Z(G)
\leq Z(L(G))$ when $G$ is a tree or when $G$ contains a Hamiltonian
path and has a certain number of edges, and it is conjectured that $Z(G) \leq
Z(L(G))$ for any non-trivial graph $G$.
Using a refinement of our proof that $B(G) \leq Z(L(G))$,
in Theorem~\ref{theorem2} we prove this conjecture.
With Theorem~\ref{theorem3} we also prove that $B(G) \leq
B(L(G))$ for any non-trivial graph. The example of the cycle shows that these bounds are tight.

\section{Some Preliminaries}

Before proving our main results we make some observations and set some notation that will be
used throughout this paper.
It is not hard to see that the zero-forcing number (resp.\ the brushing
number) of a graph is the sum of the number on the
components of the graph. Throughout this paper we primarily consider connected
graphs, but the results may be extended to disconnected graphs. Also
we do not consider the connected graph that is only a single vertex, since the line graph of this graph
is the empty graph. Note that if $G$ is connected, then $L(G)$ is also connected.

Computing the zero-forcing number for graphs is an ${\mathcal{NP}}$-hard problem
(see Theorem 3.1 in~\cite{MR3440126} and Theorems 6.3, 6.5, Corollary 6.6 in~\cite{MR3110517}). The zero-forcing number and the brushing number
are not minor monotone~\cite{MR3010007}: the operations of
edge-deletion and edge-contraction may either increase, decrease or
not change either of the zero-forcing number and the brushing number of a graph.

At each step in a zero-forcing process, one vertex $v$ forces
exactly one other vertex, say $w$, to become black; moreover, $w$ is the only vertex that $v$ is capable of forcing.
So the vertices of a graph $G$ can be arranged into $Z(G)$ oriented paths
$\sP_i$ -- these paths are called \textsl{zero-forcing chains}. The first vertex in
$\sP_i$ is a vertex in the zero-forcing set (so it is initially coloured
black) and a vertex $v$ is immediately followed by $w$ in $\sP_i$ if and only if $v$ forces
$w$. Observe that these paths are disjoint induced
paths in $G$; if a vertex never forces another vertex and is never forced, then it is the single vertex in a path of length $1$.

If the vertex $v$ forces $w$, we write $v \rightarrow w$. If $Z$ is a zero-forcing
set for $L(G)$, then the zero-forcing chains for $Z$ comprise a set of $|Z|$ induced
paths in $L(G)$. We denote these paths by
\begin{equation}\label{notationpaths}
 \begin{aligned}
\sP_{1}= & w_{1,1} \rightarrow w_{1,2} \rightarrow \dots \rightarrow w_{1,f(1)},  \\
\sP_{2}= & w_{2,1} \rightarrow w_{2,2} \rightarrow \dots \rightarrow w_{2,f(2)},  \\
  \vdots \\
\sP_{|Z|}= & w_{|Z|,1} \rightarrow w_{|Z|,2} \rightarrow \dots \rightarrow
w_{|Z|,f(|Z|)},
 \end{aligned}
\end{equation}
where $f$ is a function from $\{1, \ldots, |Z|\}$ to $\mathbb{Z}^{+}$. In the zero-forcing process, $w_{i,j}$ forces $w_{i, j+1}$ (and this
is the only vertex that $w_{i,j}$ forces). Further, we will assume that
the first time a vertex forces another vertex in the process is when $w_{1,1}$ forces $w_{1,2}$.
%place the paths of length 1 at the end of the list?

At the step in the zero-forcing process when $w_{i,j}$ forces
$w_{i,j+1}$, we say that $w_{i,j}$ is the \textsl{active} vertex. Each
vertex may perform at most one force, after which we say it is
\textsl{used}. If $w_{i,j}$ is an active vertex at some step, then
$w_{i,j+1}$ is its only white neighbour. Thus if $w_{i,j}$ is adjacent
to a vertex in another path, then the vertex in the other path must be
black at the time that $w_{i,j}$ is active.

The \textsl{brushing process} or \textsl{strategy} for a graph describes
how the brushes move through the graph, which can also be described by
listing the vertices in the order in which they fire (along with details of
how vertices with more brushes than incident dirty edges distribute their excess brushes upon firing).  Following the
route of a brush through this process would give a directed path in
$G$, but unlike zero-forcing chains these paths are neither induced nor disjoint.
\section{Brushing number of a graph vs.\ the zero-forcing number of its line graph}

%MOVE??

In this section we show that the seemingly unrelated concepts of zero-forcing and brushing
are in fact linked to one another.  Specifically we prove that the brushing number of a graph
is bounded by the zero-forcing number of its line graph.

\begin{theorem} \label{theorem1}
For any %simple
graph $G$ with no isolated vertices, $B(G)\leq Z(L(G))$.
\end{theorem}
\begin{proof}
  We may assume that $G$ is connected. %% I think it is still good to say that. Otherwise we need to explicitly say in the statement of the theorem that G is connected, but then it doesn't make much sence to say G doesn't have isolated vertices
  Let $Z$ be a zero-forcing set for $L(G)$ and let $\sP_{1}, \ldots,
  \sP_{|Z|}$
  %(as defined before)
  be the zero-forcing chains for $Z$. Note that for
  each chain $\sP_i$, the first vertex in the chain, namely $w_{i,1}$, is
  in $Z$.

  To prove the theorem, we will describe a brush configuration and
  a strategy for brushing $G$ with at most $|Z|$ brushes. For each path
  $\sP_i$, we will assign a brush to a vertex in $G$. These brushes will
  be assigned to one of the endpoints of the first edge in the path,
  carefully
  %the trick is to determine
  determining which endpoint to use at each stage.

  Initially, we can assume without loss of generality that at the
  first step in the zero-forcing process $w_{1,1} = \{a,b\}$ forces
  $w_{1,2}=\{b,c\}$. We assign the brush for $\sP_1$ to the vertex in
  $G$ that is on the edge $w_{1,1}$ but not on $w_{1,2}$, namely the
  vertex $w_{1,1} \setminus w_{1,2} = a$. Next, for $i \geq 2$,
  we add a brush to the vertex
  $w_{1,1} \cap w_{i,1}$ in $G$
  for each path $\sP_i$ for which the initial vertex
  $w_{i,1}$ is adjacent to $w_{1,1}$.
  Once this brush has been added, we
  say that $\sP_i$ has been \textsl{used}.

  At this point the vertex $a$ in $G$ can fire; the brush from $\sP_1$
  is sent to $b$, while for every other vertex adjacent to $a$ there
  has been a brush placed at $a$.

  Now we move to the first step in the zero-forcing process: the
  vertex $w_{1,1}=\{a,b\}$ forces $w_{1,2}=\{b,c\}$. There is one brush on $b$ (the
  one from $\sP_1$ initially placed on $a$) -- this brush will be sent
  down edge $w_{1,2}$. At this point vertex $b$ in $G$ can fire
  because before $a$ fires there has been a brush placed at $b$
  corresponding to each unclean edge at $b$, except for the edges
  $w_{1,1}$ and $w_{1,2}$; and the firing of $a$ decreases the number
  of unclean edges at $b$ by one while the number of brushes at $b$
  increases by one.

  We now move to the next step of the zero-forcing process, at which the active vertex $w_{i,j}=\{d,e\}$ forces $w_{i,j+1}=\{e,f\}$ (and at this stage
  $w_{i,j+1}$ is the only white neighbour of $w_{i,j}$). If $w_{i,j}$
  is adjacent to a vertex that is the head of a path $\sP_k$ that has
  not been used, say vertex $w_{k,1}$, then a brush is added to the
  vertex $w_{i,j} \cap w_{k,1}$, and we mark $\sP_k$ as used. Then in
  $G$ the vertex $d = w_{i,j} \setminus w_{i,j+1}$ (the vertex on $w_{i,j}$ that is
  not on $w_{i,j+1}$) can fire, if it has not fired before, and send a
  brush along edge $w_{i,j}$. This is because for any black vertex
  $w_{\ell,m}=\{d,g\}$ in $L(G)$ adjacent to $w_{i,j}=\{d,e\}$ that is
  not the head of a path that has not been used it must be true that
  $g$ has already fired and cleaned the edge $w_{\ell,m}=\{d,g\}$ in
  $G$ while sending a brush to $d$.  In the case that $w_{i, j+1}=\{e,f\}$ is the
  final vertex in the forcing chain $\sP_{i}$, then a similar
  argument shows that $e$ can fire in $G$ and then $f$ can fire.

  Continue like this for every step of the zero-forcing process. If
  there are any unused paths left then these paths must contain only a
  single vertex, which is initially black. If $w=\{v_1,v_2\}$ is such
  a vertex of $L(G)$ (where $v_1$ and $v_2$ are vertices of $G$) then
  add a brush to $v_1$. %(specifically, add a brush to either endpoint of $w$).
  For any other vertex $w_{k,1}$ that is adjacent to $w$ such that $\sP_k = w_{k,1}$ is an unused path, add a brush to $ w \cap w_{k,1}$. Now the vertex $v_1$ in
  $G$ can fire, followed by $v_2$.  Continue like this until all the
  paths of length $1$ are used.

  Since $G$ is connected, this process will clean all edges of $G$,
  and any vertex $v$ of $G$ that is left unclean can fire because
  there is necessarily a clean edge at $v$ (therefore at least one
  brush), but no unclean edges.  Clearly with this assignment the
  number of brushes used is equal to the number of paths $|Z|$.
\end{proof}

To illustrate the brushing strategy, we give a detailed
example. Consider the graph $G$ and its line graph $L(G)$ (see
Figure~\ref{fig:zeroforcingchains}). The set $\{b,g,h,i,j,k\}$ is a zero-forcing
set in $L(G)$ (so $b,g,h,i,j,k$ are all initially black vertices in
$L(G)$) and the zero-forcing chains are below (and are drawn in
Figure~\ref{fig:zeroforcingchains}).
\[
\sP_{1} = g \rightarrow f \rightarrow c, \quad \sP_{2}=i \rightarrow d,
\quad
 \sP_{3}=h \rightarrow e \rightarrow a,
\quad  \sP_{4}=b, \quad \sP_{5}=k, \quad \sP_{6}=j
\]
For this example, there are 11 steps in the brushing strategy. These
are given below and there is a diagram of the graph for each step in
Figure~\ref{fig:brushingexample}. In the diagram the brushes are
represented with a $\ast$ at the vertex, and at each step where a new
brush is introduced, we put a circle around the new brush.

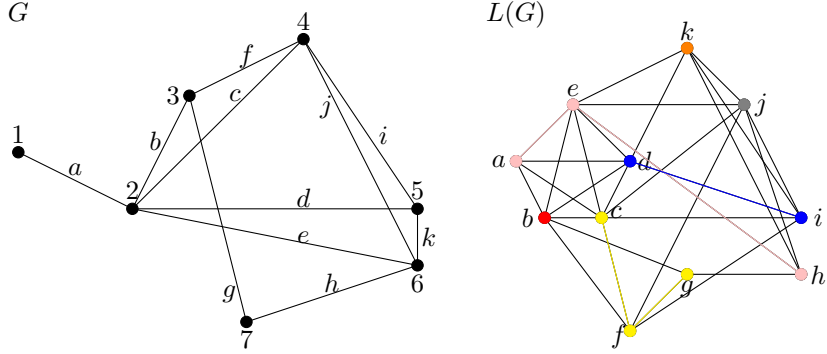
\begin{figure}
\begin{tabular}{cc}
 \begin{tikzpicture}[scale=.75]
\Ggraph
\draw [black] (-3,4.81) node  [below] {$G$};
\end{tikzpicture}
&
 \begin{tikzpicture}[scale=.75]
\LGgraph
\draw [black] (-1,5) node  [below] {$L(G)$};
\draw [thin, pink] (-1,2) to (0,3);
\draw [thin, pink] (0,3) to (4,0);
\draw [thin, yellow] (0.5,1) to (1,-1);
\draw [thin, yellow] (1,-1) to (2,0);
\draw [thin, blue] (1,2) to (4,1);
\draw [yellow,fill] (0.5,1) circle [radius=0.1];
\draw [blue,fill] (1,2) circle [radius=0.1];
\draw [pink,fill] (-1,2) circle [radius=0.1];
\draw [red,fill] (-0.5,1) circle [radius=0.1];
\draw [pink,fill] (0,3) circle [radius=0.1];
\draw [blue,fill] (4,1) circle [radius=0.1];
\draw [yellow,fill] (1,-1) circle [radius=0.1];
\draw [pink,fill] (4,0) circle [radius=0.1];
\draw [orange,fill] (2,4) circle [radius=0.1];
\draw [yellow,fill] (2,0) circle [radius=0.1];
\draw [gray,fill] (3,3) circle [radius=0.1];
\end{tikzpicture}
\end{tabular}
\caption{Zero-forcing chains for $L(G)$\label{fig:zeroforcingchains}}
\end{figure}

\begin{enumerate}[(1)]
\item In the first step of the zero-forcing process on $L(G)$ $g$
  zero forces $f$. Following our brushing strategy in $G$ we put two
  brushes at $7$ (corresponding to the initially
  black vertices $g$ and $h$ in $L(G)$).
   %1
\item Vertex $7$ fires. %2
\item Add a new brush to vertex $3$ (corresponding to the initially
  black vertex $b$).
\item Vertex $3$ fires. One brush is moved to $2$ along $b$, cleaning
  $b$; one brush is moved to $4$ along $f$, cleaning $f$. %3
\item In the next step in the zero-forcing process on $L(G)$ edge $f$
  forces $c$. Following our brushing strategy in $G$ we put two
  brushes at $4$ (corresponding to the initially black vertices $i$
  and $j$ in $L(G)$). %4
\item Vertex $4$ fires; one brush is moved to $2$ along
  $c$, cleaning $c$; one brush is moved to $5$ along
  $i$, cleaning $i$; one brush is moved to $6$ along $j$, cleaning
  $j$.  %5
\item In the next step in the zero-forcing process $i$ forces $d$. In
  $G$, edge $i$ is already clean. In order to clean $d$ in $G$
  following our brushing strategy, we put one brush at $5$
  (corresponding to the initially black vertex $k$ in $L(G)$). %6

\item Vertex 5 fires since there are two brushes at $5$ and two
  unclean edges incident with it, namely $d$ and $k$. One brush is
  moved to $2$ along $d$, cleaning $d$; one brush is moved to $6$
  along $k$, cleaning $k$.  %7

\item In the zero-forcing process in $L(G)$ vertex $h$ forces
  $e$. In $G$, edge $h$ has already been cleaned at some earlier step,
  and so at $6$ no new brushes are added. %8
   Vertex $6$ fires and cleans $e$. %9

\item Finally, in $L(G)$, vertex $e$ forces $a$. In $G$, $e$ has been
  cleaned in the previous step, and so no brushes are added. Vertex
  $2$ fires and cleans $a$. %10

\item Finally, vertex $1$ fires. %11
\end{enumerate}

\begin{figure}
\begin{tabular}{ccc}
 \begin{tikzpicture}[scale=.5]
\draw [black] (-3.5,3.5) node {(1)};
\Ggraph
\draw [black] (1,-1) node [left]{ {\circled{$\ast$}\circled{$\ast$}} };  %2 new at 7
\end{tikzpicture}
&
\begin{tikzpicture}[scale=.5]
\draw [black] (-3.5,3.5) node {(2)};
\Gtwograph
\draw [black] (0,3) node [above] {$\ast$}; %moved to 3
\draw [black] (4,0) node[right]{$\ast$}; %moved to 6
\end{tikzpicture}
&
\begin{tikzpicture}[scale=.5]
\draw [black] (-3.5,3.5) node {(3)};
\Gtwograph
\draw [black] (0,3) node [above] {$\ast$\circled{$\ast$}}; %moved to 3
\draw [black] (4,0) node[right]{$\ast$}; %moved to 6
\end{tikzpicture}
\\
\begin{tikzpicture}[scale=.5]
\draw [black] (-3.5,3.5) node {(4)};
\Gthreegraph
\draw [black] (-1,1) node [below]{$\ast$}; %moved to 2
\draw [black] (4,0) node[right]{$\ast$}; %moved to 6
\draw [black] (2,4) node [right] {$\ast$}; % moved to 4
\end{tikzpicture}
&
\begin{tikzpicture}[scale=.5]
\draw [black] (-3.5,3.5) node {(5)};
\Gfourgraph
\draw [black] (-1,1) node[below]{$\ast$}; %on 2
\draw [black] (4,0) node[right]{$\ast$}; %moved to 6
\draw [black] (2,4) node [right]{$\ast$ \circled{$\ast$}\circled{$\ast$}}; % on 4
\end{tikzpicture}
&
\begin{tikzpicture}[scale=.5]
\draw [black] (-3.5,3.5) node {(6)};
\Gfivegraph
\draw [black] (-1,1) node  [below] {$\ast\ast$}; %2 on 2
\draw [black] (4,0) node [right] {$\ast\ast$}; % 2 on6
\draw [black] (4,1) node [right] {$\ast$}; %1 on 5
\draw [black] (2,4) circle [radius=0.1];
\end{tikzpicture}
\\
\begin{tikzpicture}[scale=.5]
\draw [black] (-3.5,3.5) node {(7)};
\Gsixgraph
\draw [black] (-1,1) node  [below] {$\ast\ast$}; %2 on 2
\draw [black] (4,0) node [right] {$\ast\ast$}; % 2 on6
\draw [black] (4,1) node [right] {$\ast$\circled{$\ast$}}; %1 on 5
\draw [black] (2,4) circle [radius=0.1];
\end{tikzpicture}
&
\begin{tikzpicture}[scale=.5]
\draw [black] (-3.5,3.5) node {(8)};
\Gsevengraph
\draw [black] (-1,1) node  [below] {$\ast\ast\ast$}; %2 on 2
\draw [black] (4,0) node [right] {$\ast\ast \ast$}; % 2 on6
\draw [black] (2,4) circle [radius=0.1];
\draw [black] (4,1) circle [radius=0.1];

\end{tikzpicture}
&
\begin{tikzpicture}[scale=.5]
\draw [black] (-3.5,3.5) node {(9)};
\Gninegraph
\draw [black] (-1,1) node  [below] {$\ast\ast\ast\ast$}; %4 on 2
\draw [black] (4,0) node [right] {$\ast\ast$}; % 2 on6
\draw [black] (2,4) circle [radius=0.1];
\draw [black] (4,1) circle [radius=0.1];
\draw [black] (4,0) circle [radius=0.1];
\end{tikzpicture}
 \\
\begin{tikzpicture}[scale=.5]
\draw [black] (-3.5,3.5) node {(10)};
\Gtengraph
\draw [black] (-1,1) node  [below] {$\ast\ast\ast$}; %2 on 2
\draw [black] (-3,2) node [below] {$\ast$}; % 1 on6
\draw [black] (4,0) node [right] {$\ast\ast$};

\draw [black] (2,4) circle [radius=0.1];
\draw [black] (4,1) circle [radius=0.1];
\draw [black] (4,0) circle [radius=0.1];
\draw [black] (-1,1) circle [radius=0.1];

\end{tikzpicture}

&\begin{tikzpicture}[scale=.5]
\draw [black] (-3.5,3.5) node {(11)};
\Gelevengraph
\draw [black] (-1,1) node  [below] {$\ast\ast\ast$}; %2 on 2
\draw [black] (-3,2) node [below] {$\ast$}; % 1 on6
\draw [black] (4,0) node [right] {$\ast\ast$};

\draw [black] (2,4) circle [radius=0.1];
\draw [black] (4,1) circle [radius=0.1];
\draw [black] (4,0) circle [radius=0.1];
\draw [black] (-1,1) circle [radius=0.1];
\draw [black] (-3,2) circle [radius=0.1];
\end{tikzpicture}
& \\
\end{tabular}
\caption{Example to illustrate the brushing strategy}\label{fig:brushingexample}

\end{figure}
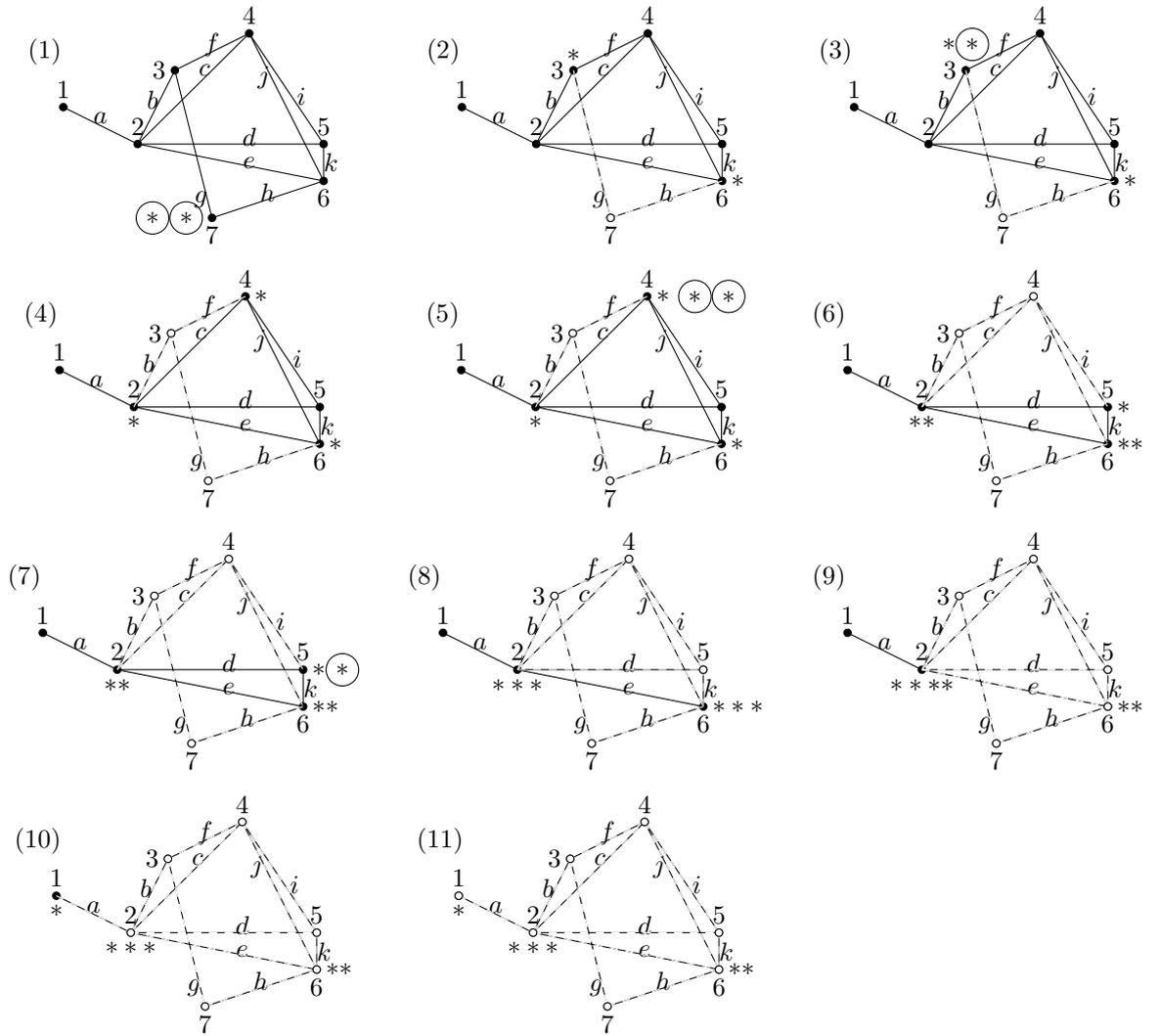

Theorem~\ref{theorem1} is stated for graphs without isolated
vertices. Adding an isolated vertex to a graph increases the brushing
number by one, since one brush must be placed on the isolated vertex
to clean it. However adding an isolated vertex to a graph does not change
the line graph. So it is easy to see that the following corollary holds.

\begin{corollary}\label{cor:isolateone}
For any graph $G$ with $k$ isolated vertices, $B(G) \leq Z(L(G)) + k$.
\end{corollary}

Recall that $b(G)$ is defined to be the minimum number of
brushes needed to clean all edges and vertices of $G$ where each time
a vertex fires only one brush is allowed to be moved along each
incident edge. It is not difficult to
observe that the brushing strategy given in the proof of Theorem~\ref{theorem1} never requires more than one brush to be moved along an
edge in $G$. This observation results in the following corollary to
Theorem~\ref{theorem1}.

\begin{corollary} \label{corollary}
For any graph $G$ with no isolated vertices, $B(G) \leq b(G)\leq Z(L(G))$.
\end{corollary}

Similar to Corollary \ref{cor:isolateone}, we also get the following result.

\begin{corollary}
For any graph $G$ with $k$ isolated vertices, $B(G) \leq b(G) \leq Z(L(G)) + k$.
\end{corollary}

We note that often the strategy described in Theorem~\ref{theorem1}
provides a brushing of $G$ that could sometimes use %% I've  changed "more than $B(G)$" back to "strictly less than $Z(L(G))$". This is because when we apply our brushing strategy step by step we see that in many examples we don't need to introduce a brush corresponding to some vertices in the zero forcing set.
strictly less than $Z(L(G))$ brushes (by adding a brush to a vertex at some step in the brushing procedure only if it is necessary to do so to make the vertex fire). However, the brushing strategy in Theorem~\ref{theorem1} does
not give a clear insight on how small $B(G)$ can be compared to
$Z(L(G))$, nor on the problem of characterizing the graphs $G$ for
which $B(G)= Z(L(G))$. We note that equality holds for cycles and
paths. Furthermore by Equation~\ref{eq:complete}, $Z(L(K_n)) \leq 2B(K_n)$
(for $n \geq 3$) and by Equation~\ref{eq:star}, $Z(L(K_{1,n}))+1 \leq
2B(K_{1,n})$ for all $n$. From these examples, one might also wonder if it is
possible to bound $Z(L(G))$ from above by some multiple of $B(G)$. In
what follows we construct a family of graphs to show that this is not
the case.

%KAREN look at this!!
\begin{theorem}\label{thm:NoReal}
There exists no real number $c$ for which  $Z(L(G)) \leq cB(G)$ for all graphs $G$.
\end{theorem}

\begin{proof}
  Consider the Cartesian product graph $P_{r} \square C_{s}$ where
  $P_{r}$ is the path with $r$ vertices and $C_{s}$ is the cycle with
  $s$ vertices. For what follows we represent $P_{r} \square C_{s}$ as
  the graph made up of $r$ concentric $s$-cycles and $s(r-1)$
  additional edges joining the corresponding vertices of the
  cycles. For an example with $r=3$ and $s=4$, see the first graph in
  Figure~\ref{fig:brushing}.

  First we prove that $B(P_{r} \square C_{s}) \leq s+2$. This is easy
  to establish by starting with $s+2$ brushes at some vertex on the
  outermost concentric cycle of $P_{r} \square C_{s}$. Each time that a brush fires,
  if it has more brushes than incident dirty edges, then we send all excess brushes together to
  the nearest perimeter vertex in the clockwise direction.
  In Figure~\ref{fig:brushing} we present the strategy on $P_{3} \square C_{4}$,
  and note that this strategy easily generalizes for any $r \geq 1, s
  \geq 3$.

  Next we prove that $Z(L(P_{r} \square C_{s})) \geq r-1$. $L(P_{r}
  \square C_{s})$ can be represented as a graph that consists of $r-1$
  layers each having $2s$ vertices and $6s$ edges, except for the
  innermost layer which has $3s$ vertices and $6s$ edges. (For an
  example, see Figure~\ref{fig:layers}). Now suppose that initially
  there is a layer $k$ that has no black vertices. We may suppose that
  initially all other vertices are black. Unless layer $k$ is the
  innermost layer, half of the vertices of layer $k$ are adjacent to
  some vertex from a more central layer, and they can indeed be forced
  to black by those vertices. We note that these $s$ vertices of layer
  $k$ which are just forced to black are each adjacent to two of the
  remaining $s$ (white) vertices of layer $k$, and hence they cannot
  force any of these white vertices to black. Similarly no vertex from
  any outside layer can force any of these $s$ white vertices to black
  because any vertex from an outer layer is adjacent to either zero or two
  such white vertices. Therefore each layer of $L(P_{r} \square
  C_{s})$ has to have at least one black vertex initially, and
  $Z(L(P_{r} \square C_{s})) \geq r-1$ follows.

  So for any real number $c$ and
  each $c' > c$ such that $c'(s+2)+2$ is a positive integer,
  $G=P_{c'(s+2)+2} \square C_{s}$ ($s \geq 3$ an integer) yields an
  example with $cB(G) < Z(L(G))$.
\end{proof}

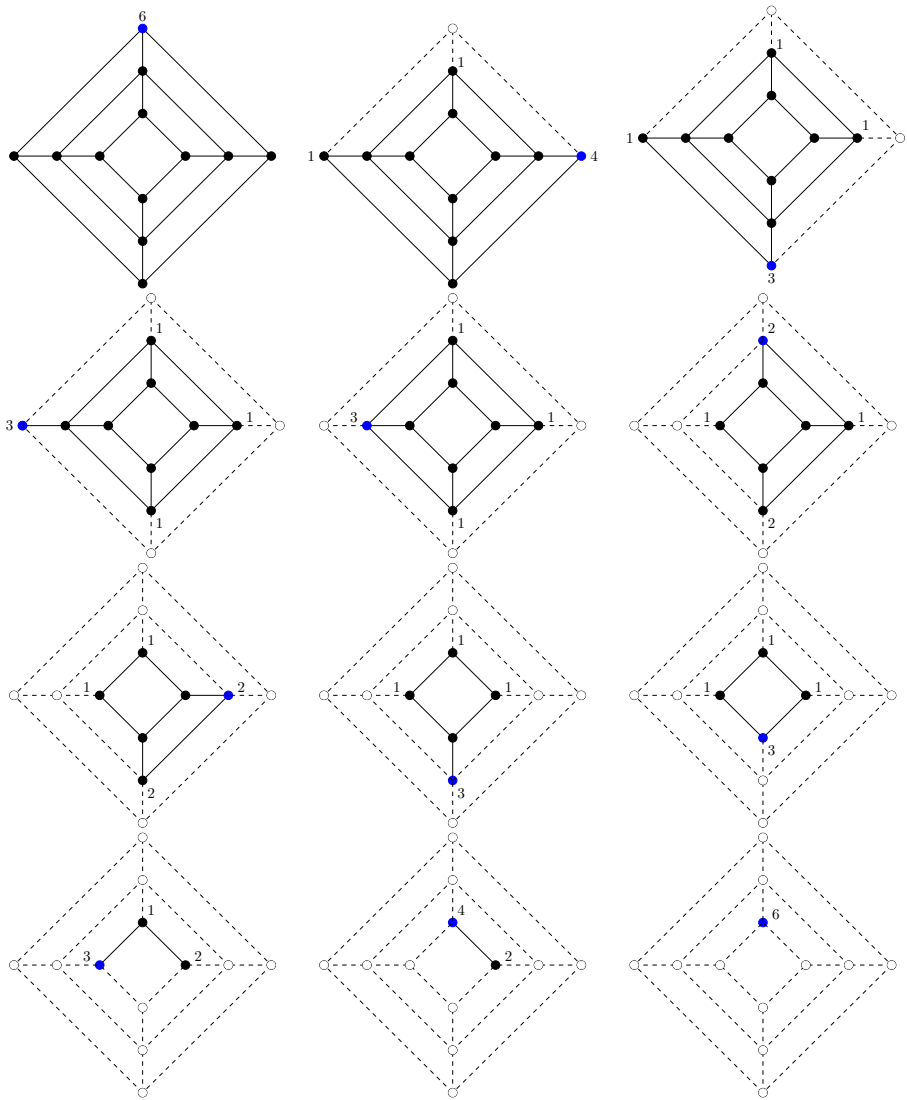
\begin{figure}
\begin{adjustbox}{max width=\textwidth}
\begin{tabular}{ccc}

\begin{tikzpicture}
\draw [black,fill] (-1,0) circle [radius=0.1];
\draw [black,fill] (1,0) circle [radius=0.1];
\draw [black,fill] (0,-1) circle [radius=0.1];
\draw [black,fill] (0,1) circle [radius=0.1];

\draw [black,fill] (-2,0) circle [radius=0.1];
\draw [black,fill] (2,0) circle [radius=0.1];
\draw [black,fill] (0,-2) circle [radius=0.1];
\draw [black,fill] (0,2) circle [radius=0.1];

\draw [black,fill] (-3,0) circle [radius=0.1];
\draw [black,fill] (3,0) circle [radius=0.1];
\draw [black,fill] (0,-3) circle [radius=0.1];
\draw [black] (0,3.3) node {$6$};

\draw [thin] (-3,0) to (-1,0);
\draw [thin] (1,0) to (3,0);
\draw [thin] (0,-1) to (0,-3);
\draw [thin] (0,1) to (0,3);

\draw [thin] (-3,0) to (0,3);
\draw [thin] (-2,0) to (0,2);
\draw [thin] (-1,0) to (0,1);

\draw [thin] (-3,0) to (0,-3);
\draw [thin] (-2,0) to (0,-2);
\draw [thin] (-1,0) to (0,-1);

\draw [thin] (3,0) to (0,-3);
\draw [thin] (2,0) to (0,-2);
\draw [thin] (1,0) to (0,-1);

\draw [thin] (3,0) to (0,3);
\draw [thin] (2,0) to (0,2);
\draw [thin] (1,0) to (0,1);

\draw [blue,fill] (0,3) circle [radius=0.1];
\end{tikzpicture}
  &
\begin{tikzpicture}
\draw [black,fill] (-1,0) circle [radius=0.1];
\draw [black,fill] (1,0) circle [radius=0.1];
\draw [black,fill] (0,-1) circle [radius=0.1];
\draw [black,fill] (0,1) circle [radius=0.1];

\draw [black,fill] (-2,0) circle [radius=0.1];
\draw [black,fill] (2,0) circle [radius=0.1];
\draw [black,fill] (0,-2) circle [radius=0.1];
\draw [black,fill] (0,2) circle [radius=0.1];
\draw [black] (0.2,2.2) node {$1$};

\draw [black,fill] (-3,0) circle [radius=0.1];
\draw [black] (-3.3,0) node {$1$};
\draw [black] (3.3,0) node {$4$};
\draw [black,fill] (0,-3) circle [radius=0.1];

\draw [thin] (-3,0) to (-1,0);
\draw [thin] (1,0) to (3,0);
\draw [thin] (0,-1) to (0,-3);

\draw [thin] (-2,0) to (0,2);
\draw [thin] (-1,0) to (0,1);

\draw [thin] (-3,0) to (0,-3);
\draw [thin] (-2,0) to (0,-2);
\draw [thin] (-1,0) to (0,-1);

\draw [thin] (3,0) to (0,-3);
\draw [thin] (2,0) to (0,-2);
\draw [thin] (1,0) to (0,-1);

\draw [thin] (2,0) to (0,2);
\draw [thin] (1,0) to (0,1);

\draw [thin] (0,1) to (0,2);
\draw [blue,fill] (3,0) circle [radius=0.1];

\draw [thin, dashed] (0,2) to (0,3);
\draw [thin, dashed] (0,3) to (-3,0);
\draw [thin, dashed] (0,3) to (3,0);
\draw [black,fill] (0,3) circle [radius=0.1];
\draw [white,fill] (0,3) circle [radius=0.09];

\end{tikzpicture}
&
\begin{tikzpicture}
\draw [black,fill] (-1,0) circle [radius=0.1];
\draw [black,fill] (1,0) circle [radius=0.1];
\draw [black,fill] (0,-1) circle [radius=0.1];
\draw [black,fill] (0,1) circle [radius=0.1];

\draw [black,fill] (-2,0) circle [radius=0.1];
\draw [black,fill] (2,0) circle [radius=0.1];
\draw [black] (2.2,0.3) node {$1$};
\draw [black,fill] (0,-2) circle [radius=0.1];
\draw [black,fill] (0,2) circle [radius=0.1];
\draw [black] (0.2,2.2) node {$1$};

\draw [black,fill] (-3,0) circle [radius=0.1];
\draw [black] (-3.3,0) node {$1$};

%\draw [black,fill] (3,0) circle [radius=0.1];
%\draw [black] (3.3,0) node {$4$};
%\draw [black,fill] (0,3) circle [radius=0.1];

\draw [black] (0,-3.3) node {$3$};

\draw [thin] (-3,0) to (-1,0);
\draw [thin] (1,0) to (2,0);
\draw [thin] (0,-1) to (0,-3);
%\draw [thin] (0,1) to (0,3);

%\draw [thin] (-3,0) to (0,3);
\draw [thin] (-2,0) to (0,2);
\draw [thin] (-1,0) to (0,1);

\draw [thin] (-3,0) to (0,-3);
\draw [thin] (-2,0) to (0,-2);
\draw [thin] (-1,0) to (0,-1);

%\draw [thin] (3,0) to (0,-3);
\draw [thin] (2,0) to (0,-2);
\draw [thin] (1,0) to (0,-1);

%\draw [thin] (3,0) to (0,3);
\draw [thin] (2,0) to (0,2);
\draw [thin] (1,0) to (0,1);

%      \draw [black,solid] (-5,3.5) node {$i$};
%      \draw [black] (-3.25,4) node {$0$};
\draw [thin] (0,1) to (0,2);
\draw [blue,fill] (0,-3) circle [radius=0.1];

\draw [thin, dashed] (3,0) to (2,0);
\draw [thin, dashed] (3,0) to (0,-3);
\draw [thin, dashed] (0,2) to (0,3);
\draw [thin, dashed] (0,3) to (-3,0);
\draw [thin, dashed] (0,3) to (3,0);
\draw [black,fill] (0,3) circle [radius=0.1];
\draw [white,fill] (0,3) circle [radius=0.09];
\draw [black,fill] (3,0) circle [radius=0.1];
\draw [white,fill] (3,0) circle [radius=0.09];
\end{tikzpicture}
\\
\begin{tikzpicture}
\draw [black,fill] (-1,0) circle [radius=0.1];
\draw [black,fill] (1,0) circle [radius=0.1];
\draw [black,fill] (0,-1) circle [radius=0.1];
\draw [black,fill] (0,1) circle [radius=0.1];

\draw [black,fill] (-2,0) circle [radius=0.1];
\draw [black,fill] (2,0) circle [radius=0.1];
\draw [black] (2.3,0.2) node {$1$};
\draw [black,fill] (0,-2) circle [radius=0.1];
\draw [black] (0.2,-2.3) node {$1$};
\draw [black,fill] (0,2) circle [radius=0.1];
\draw [black] (0.2,2.3) node {$1$};

\draw [blue,fill] (-3,0) circle [radius=0.1];
\draw [black] (-3.3,0) node {$3$};

\draw [thin] (-3,0) to (-1,0);
\draw [thin] (1,0) to (2,0);
\draw [thin] (0,-1) to (0,-2);

\draw [thin] (-2,0) to (0,2);
\draw [thin] (-1,0) to (0,1);

\draw [thin] (-2,0) to (0,-2);
\draw [thin] (-1,0) to (0,-1);

\draw [thin] (2,0) to (0,-2);
\draw [thin] (1,0) to (0,-1);

\draw [thin] (2,0) to (0,2);
\draw [thin] (1,0) to (0,1);
\draw [thin] (0,1) to (0,2);
\draw [blue,fill] (-3,0) circle [radius=0.1];

\draw [thin, dashed] (0,-3) to (0,-2);
\draw [thin, dashed] (0,-3) to (-3,0);
\draw [thin, dashed] (3,0) to (2,0);
\draw [thin, dashed] (3,0) to (0,-3);
\draw [thin, dashed] (0,2) to (0,3);
\draw [thin, dashed] (0,3) to (-3,0);
\draw [thin, dashed] (0,3) to (3,0);
\draw [black,fill] (0,3) circle [radius=0.1];
\draw [white,fill] (0,3) circle [radius=0.09];
\draw [black,fill] (3,0) circle [radius=0.1];
\draw [white,fill] (3,0) circle [radius=0.09];
\draw [black,fill] (0,-3) circle [radius=0.1];
\draw [white,fill] (0,-3) circle [radius=0.09];
\end{tikzpicture}
&
\begin{tikzpicture}
\draw [black,fill] (-1,0) circle [radius=0.1];
\draw [black,fill] (1,0) circle [radius=0.1];
\draw [black,fill] (0,-1) circle [radius=0.1];
\draw [black,fill] (0,1) circle [radius=0.1];

\draw [black] (-2.3,0.2) node {$3$};
\draw [black,fill] (2,0) circle [radius=0.1];
\draw [black] (2.3,0.2) node {$1$};
\draw [black,fill] (0,-2) circle [radius=0.1];
\draw [black] (0.2,-2.3) node {$1$};
\draw [black,fill] (0,2) circle [radius=0.1];
\draw [black] (0.2,2.3) node {$1$};

\draw [thin] (-2,0) to (-1,0);
\draw [thin] (1,0) to (2,0);
\draw [thin] (0,-1) to (0,-2);

\draw [thin] (-2,0) to (0,2);
\draw [thin] (-1,0) to (0,1);

\draw [thin] (-2,0) to (0,-2);
\draw [thin] (-1,0) to (0,-1);

\draw [thin] (2,0) to (0,-2);
\draw [thin] (1,0) to (0,-1);

\draw [thin] (2,0) to (0,2);
\draw [thin] (1,0) to (0,1);
\draw [thin] (0,1) to (0,2);
\draw [blue,fill] (-2,0) circle [radius=0.1];

\draw [thin, dashed] (-3,0) to (-2,0);
\draw [thin, dashed] (0,-3) to (0,-2);
\draw [thin, dashed] (0,-3) to (-3,0);
\draw [thin, dashed] (3,0) to (2,0);
\draw [thin, dashed] (3,0) to (0,-3);
\draw [thin, dashed] (0,2) to (0,3);
\draw [thin, dashed] (0,3) to (-3,0);
\draw [thin, dashed] (0,3) to (3,0);
\draw [black,fill] (0,3) circle [radius=0.1];
\draw [white,fill] (0,3) circle [radius=0.09];
\draw [black,fill] (3,0) circle [radius=0.1];
\draw [white,fill] (3,0) circle [radius=0.09];
\draw [black,fill] (0,-3) circle [radius=0.1];
\draw [white,fill] (0,-3) circle [radius=0.09];
\draw [black,fill] (-3,0) circle [radius=0.1];
\draw [white,fill] (-3,0) circle [radius=0.09];
\end{tikzpicture}
&
\begin{tikzpicture}
\draw [black,fill] (-1,0) circle [radius=0.1];
\draw [black] (-1.3,0.2) node {$1$};
\draw [black,fill] (1,0) circle [radius=0.1];
\draw [black,fill] (0,-1) circle [radius=0.1];
\draw [black,fill] (0,1) circle [radius=0.1];

\draw [black,fill] (2,0) circle [radius=0.1];
\draw [black] (2.3,0.2) node {$1$};
\draw [black,fill] (0,-2) circle [radius=0.1];
\draw [black] (0.2,-2.3) node {$2$};
\draw [black] (0.2,2.3) node {$2$};

\draw [thin] (1,0) to (2,0);
\draw [thin] (0,-1) to (0,-2);
\draw [thin] (0,1) to (0,2);
\draw [thin] (-1,0) to (0,1);

\draw [thin] (-1,0) to (0,-1);

\draw [thin] (2,0) to (0,-2);
\draw [thin] (1,0) to (0,-1);

\draw [thin] (2,0) to (0,2);
\draw [thin] (1,0) to (0,1);
\draw [blue,fill] (0,2) circle [radius=0.1];

\draw [thin, dashed] (-2,0) to (-1,0);
\draw [thin, dashed] (-2,0) to (0,2);
\draw [thin, dashed] (-2,0) to (0,-2);
\draw [thin, dashed] (-3,0) to (-2,0);
\draw [thin, dashed] (0,-3) to (0,-2);
\draw [thin, dashed] (0,-3) to (-3,0);
\draw [thin, dashed] (3,0) to (2,0);
\draw [thin, dashed] (3,0) to (0,-3);
\draw [thin, dashed] (0,2) to (0,3);
\draw [thin, dashed] (0,3) to (-3,0);
\draw [thin, dashed] (0,3) to (3,0);
\draw [black,fill] (0,3) circle [radius=0.1];
\draw [white,fill] (0,3) circle [radius=0.09];
\draw [black,fill] (3,0) circle [radius=0.1];
\draw [white,fill] (3,0) circle [radius=0.09];
\draw [black,fill] (0,-3) circle [radius=0.1];
\draw [white,fill] (0,-3) circle [radius=0.09];
\draw [black,fill] (-3,0) circle [radius=0.1];
\draw [white,fill] (-3,0) circle [radius=0.09];
\draw [black,fill] (-2,0) circle [radius=0.1];
\draw [white,fill] (-2,0) circle [radius=0.09];

\end{tikzpicture}
\\
\begin{tikzpicture}
\draw [black,fill] (-1,0) circle [radius=0.1];
\draw [black] (-1.3,0.2) node {$1$};
\draw [black,fill] (1,0) circle [radius=0.1];
\draw [black,fill] (0,-1) circle [radius=0.1];
\draw [black,fill] (0,1) circle [radius=0.1];
\draw [black] (0.2,1.3) node {$1$};

\draw [black] (2.3,0.2) node {$2$};
\draw [black,fill] (0,-2) circle [radius=0.1];
\draw [black] (0.2,-2.3) node {$2$};

\draw [thin] (1,0) to (2,0);
\draw [thin] (0,-1) to (0,-2);
\draw [thin] (-1,0) to (0,1);

\draw [thin] (-1,0) to (0,-1);

\draw [thin] (2,0) to (0,-2);
\draw [thin] (1,0) to (0,-1);

\draw [thin] (1,0) to (0,1);
\draw [blue,fill] (2,0) circle [radius=0.1];

\draw [thin, dashed] (0,2) to (0,1);
\draw [thin, dashed] (0,2) to (2,0);
\draw [thin, dashed] (-2,0) to (-1,0);
\draw [thin, dashed] (-2,0) to (0,2);
\draw [thin, dashed] (-2,0) to (0,-2);
\draw [thin, dashed] (-3,0) to (-2,0);
\draw [thin, dashed] (0,-3) to (0,-2);
\draw [thin, dashed] (0,-3) to (-3,0);
\draw [thin, dashed] (3,0) to (2,0);
\draw [thin, dashed] (3,0) to (0,-3);
\draw [thin, dashed] (0,2) to (0,3);
\draw [thin, dashed] (0,3) to (-3,0);
\draw [thin, dashed] (0,3) to (3,0);
\draw [black,fill] (0,3) circle [radius=0.1];
\draw [white,fill] (0,3) circle [radius=0.09];
\draw [black,fill] (3,0) circle [radius=0.1];
\draw [white,fill] (3,0) circle [radius=0.09];
\draw [black,fill] (0,-3) circle [radius=0.1];
\draw [white,fill] (0,-3) circle [radius=0.09];
\draw [black,fill] (-3,0) circle [radius=0.1];
\draw [white,fill] (-3,0) circle [radius=0.09];
\draw [black,fill] (-2,0) circle [radius=0.1];
\draw [white,fill] (-2,0) circle [radius=0.09];
\draw [black,fill] (0,2) circle [radius=0.1];
\draw [white,fill] (0,2) circle [radius=0.09];
\end{tikzpicture}
&
\begin{tikzpicture}
\draw [black,fill] (-1,0) circle [radius=0.1];
\draw [black] (-1.3,0.2) node {$1$};
\draw [black,fill] (1,0) circle [radius=0.1];
\draw [black] (1.3,0.2) node {$1$};

\draw [black,fill] (0,-1) circle [radius=0.1];
\draw [black,fill] (0,1) circle [radius=0.1];
\draw [black] (0.2,1.3) node {$1$};

\draw [black] (0.2,-2.3) node {$3$};

\draw [thin] (0,-1) to (0,-2);
\draw [thin] (-1,0) to (0,1);

\draw [thin] (-1,0) to (0,-1);

\draw [thin] (1,0) to (0,-1);

\draw [thin] (1,0) to (0,1);
\draw [blue,fill] (0,-2) circle [radius=0.1];

\draw [thin, dashed] (2,0) to (1,0);
\draw [thin, dashed] (2,0) to (0,-2);
\draw [thin, dashed] (0,2) to (0,1);
\draw [thin, dashed] (0,2) to (2,0);
\draw [thin, dashed] (-2,0) to (-1,0);
\draw [thin, dashed] (-2,0) to (0,2);
\draw [thin, dashed] (-2,0) to (0,-2);
\draw [thin, dashed] (-3,0) to (-2,0);
\draw [thin, dashed] (0,-3) to (0,-2);
\draw [thin, dashed] (0,-3) to (-3,0);
\draw [thin, dashed] (3,0) to (2,0);
\draw [thin, dashed] (3,0) to (0,-3);
\draw [thin, dashed] (0,2) to (0,3);
\draw [thin, dashed] (0,3) to (-3,0);
\draw [thin, dashed] (0,3) to (3,0);
\draw [black,fill] (0,3) circle [radius=0.1];
\draw [white,fill] (0,3) circle [radius=0.09];
\draw [black,fill] (3,0) circle [radius=0.1];
\draw [white,fill] (3,0) circle [radius=0.09];
\draw [black,fill] (0,-3) circle [radius=0.1];
\draw [white,fill] (0,-3) circle [radius=0.09];
\draw [black,fill] (-3,0) circle [radius=0.1];
\draw [white,fill] (-3,0) circle [radius=0.09];
\draw [black,fill] (-2,0) circle [radius=0.1];
\draw [white,fill] (-2,0) circle [radius=0.09];
\draw [black,fill] (0,2) circle [radius=0.1];
\draw [white,fill] (0,2) circle [radius=0.09];
\draw [black,fill] (2,0) circle [radius=0.1];
\draw [white,fill] (2,0) circle [radius=0.09];
\end{tikzpicture}
&
\begin{tikzpicture}
\draw [black,fill] (-1,0) circle [radius=0.1];
\draw [black] (-1.3,0.2) node {$1$};
\draw [black,fill] (1,0) circle [radius=0.1];
\draw [black] (1.3,0.2) node {$1$};

\draw [black] (0.2,-1.3) node {$3$};

\draw [black,fill] (0,1) circle [radius=0.1];
\draw [black] (0.2,1.3) node {$1$};

\draw [thin] (-1,0) to (0,1);

\draw [thin] (-1,0) to (0,-1);

\draw [thin] (1,0) to (0,-1);

\draw [thin] (1,0) to (0,1);
\draw [blue,fill] (0,-1) circle [radius=0.1];

\draw [thin, dashed] (0,-2) to (0,-1);
\draw [thin, dashed] (2,0) to (1,0);
\draw [thin, dashed] (2,0) to (0,-2);
\draw [thin, dashed] (0,2) to (0,1);
\draw [thin, dashed] (0,2) to (2,0);
\draw [thin, dashed] (-2,0) to (-1,0);
\draw [thin, dashed] (-2,0) to (0,2);
\draw [thin, dashed] (-2,0) to (0,-2);
\draw [thin, dashed] (-3,0) to (-2,0);
\draw [thin, dashed] (0,-3) to (0,-2);
\draw [thin, dashed] (0,-3) to (-3,0);
\draw [thin, dashed] (3,0) to (2,0);
\draw [thin, dashed] (3,0) to (0,-3);
\draw [thin, dashed] (0,2) to (0,3);
\draw [thin, dashed] (0,3) to (-3,0);
\draw [thin, dashed] (0,3) to (3,0);
\draw [black,fill] (0,3) circle [radius=0.1];
\draw [white,fill] (0,3) circle [radius=0.09];
\draw [black,fill] (3,0) circle [radius=0.1];
\draw [white,fill] (3,0) circle [radius=0.09];
\draw [black,fill] (0,-3) circle [radius=0.1];
\draw [white,fill] (0,-3) circle [radius=0.09];
\draw [black,fill] (-3,0) circle [radius=0.1];
\draw [white,fill] (-3,0) circle [radius=0.09];
\draw [black,fill] (-2,0) circle [radius=0.1];
\draw [white,fill] (-2,0) circle [radius=0.09];
\draw [black,fill] (0,2) circle [radius=0.1];
\draw [white,fill] (0,2) circle [radius=0.09];
\draw [black,fill] (2,0) circle [radius=0.1];
\draw [white,fill] (2,0) circle [radius=0.09];
\draw [black,fill] (0,-2) circle [radius=0.1];
\draw [white,fill] (0,-2) circle [radius=0.09];
\end{tikzpicture}
\\
\begin{tikzpicture}
\draw [black] (-1.3,0.2) node {$3$};
\draw [black,fill] (1,0) circle [radius=0.1];
\draw [black] (1.3,0.2) node {$2$};

\draw [black,fill] (0,1) circle [radius=0.1];
\draw [black] (0.2,1.3) node {$1$};

\draw [thin] (-1,0) to (0,1);

\draw [thin] (1,0) to (0,1);
\draw [blue,fill] (-1,0) circle [radius=0.1];

\draw [thin, dashed] (0,-1) to (-1,0);
\draw [thin, dashed] (0,-1) to (1,0);
\draw [thin, dashed] (0,-2) to (0,-1);
\draw [thin, dashed] (2,0) to (1,0);
\draw [thin, dashed] (2,0) to (0,-2);
\draw [thin, dashed] (0,2) to (0,1);
\draw [thin, dashed] (0,2) to (2,0);
\draw [thin, dashed] (-2,0) to (-1,0);
\draw [thin, dashed] (-2,0) to (0,2);
\draw [thin, dashed] (-2,0) to (0,-2);
\draw [thin, dashed] (-3,0) to (-2,0);
\draw [thin, dashed] (0,-3) to (0,-2);
\draw [thin, dashed] (0,-3) to (-3,0);
\draw [thin, dashed] (3,0) to (2,0);
\draw [thin, dashed] (3,0) to (0,-3);
\draw [thin, dashed] (0,2) to (0,3);
\draw [thin, dashed] (0,3) to (-3,0);
\draw [thin, dashed] (0,3) to (3,0);
\draw [black,fill] (0,3) circle [radius=0.1];
\draw [white,fill] (0,3) circle [radius=0.09];
\draw [black,fill] (3,0) circle [radius=0.1];
\draw [white,fill] (3,0) circle [radius=0.09];
\draw [black,fill] (0,-3) circle [radius=0.1];
\draw [white,fill] (0,-3) circle [radius=0.09];
\draw [black,fill] (-3,0) circle [radius=0.1];
\draw [white,fill] (-3,0) circle [radius=0.09];
\draw [black,fill] (-2,0) circle [radius=0.1];
\draw [white,fill] (-2,0) circle [radius=0.09];
\draw [black,fill] (0,2) circle [radius=0.1];
\draw [white,fill] (0,2) circle [radius=0.09];
\draw [black,fill] (2,0) circle [radius=0.1];
\draw [white,fill] (2,0) circle [radius=0.09];
\draw [black,fill] (0,-2) circle [radius=0.1];
\draw [white,fill] (0,-2) circle [radius=0.09];
\draw [black,fill] (0,-1) circle [radius=0.1];
\draw [white,fill] (0,-1) circle [radius=0.09];
\end{tikzpicture}
&
\begin{tikzpicture}
\draw [black,fill] (1,0) circle [radius=0.1];
\draw [black] (1.3,0.2) node {$2$};

\draw [black] (0.2,1.3) node {$4$};

\draw [thin] (1,0) to (0,1);
\draw [blue,fill] (0,1) circle [radius=0.1];

\draw [thin, dashed] (-1,0) to (0,1);
\draw [thin, dashed] (0,-1) to (-1,0);
\draw [thin, dashed] (0,-1) to (1,0);
\draw [thin, dashed] (0,-2) to (0,-1);
\draw [thin, dashed] (2,0) to (1,0);
\draw [thin, dashed] (2,0) to (0,-2);
\draw [thin, dashed] (0,2) to (0,1);
\draw [thin, dashed] (0,2) to (2,0);
\draw [thin, dashed] (-2,0) to (-1,0);
\draw [thin, dashed] (-2,0) to (0,2);
\draw [thin, dashed] (-2,0) to (0,-2);
\draw [thin, dashed] (-3,0) to (-2,0);
\draw [thin, dashed] (0,-3) to (0,-2);
\draw [thin, dashed] (0,-3) to (-3,0);
\draw [thin, dashed] (3,0) to (2,0);
\draw [thin, dashed] (3,0) to (0,-3);
\draw [thin, dashed] (0,2) to (0,3);
\draw [thin, dashed] (0,3) to (-3,0);
\draw [thin, dashed] (0,3) to (3,0);
\draw [black,fill] (0,3) circle [radius=0.1];
\draw [white,fill] (0,3) circle [radius=0.09];
\draw [black,fill] (3,0) circle [radius=0.1];
\draw [white,fill] (3,0) circle [radius=0.09];
\draw [black,fill] (0,-3) circle [radius=0.1];
\draw [white,fill] (0,-3) circle [radius=0.09];
\draw [black,fill] (-3,0) circle [radius=0.1];
\draw [white,fill] (-3,0) circle [radius=0.09];
\draw [black,fill] (-2,0) circle [radius=0.1];
\draw [white,fill] (-2,0) circle [radius=0.09];
\draw [black,fill] (0,2) circle [radius=0.1];
\draw [white,fill] (0,2) circle [radius=0.09];
\draw [black,fill] (2,0) circle [radius=0.1];
\draw [white,fill] (2,0) circle [radius=0.09];
\draw [black,fill] (0,-2) circle [radius=0.1];
\draw [white,fill] (0,-2) circle [radius=0.09];
\draw [black,fill] (0,-1) circle [radius=0.1];
\draw [white,fill] (0,-1) circle [radius=0.09];
\draw [black,fill] (-1,0) circle [radius=0.1];
\draw [white,fill] (-1,0) circle [radius=0.09];
\end{tikzpicture}
&
\begin{tikzpicture}
\draw [black] (0.3,1.2) node {$6$};
\draw [blue,fill] (0,1) circle [radius=0.1];

\draw [thin, dashed] (1,0) to (0,1);
\draw [thin, dashed] (-1,0) to (0,1);
\draw [thin, dashed] (0,-1) to (-1,0);
\draw [thin, dashed] (0,-1) to (1,0);
\draw [thin, dashed] (0,-2) to (0,-1);
\draw [thin, dashed] (2,0) to (1,0);
\draw [thin, dashed] (2,0) to (0,-2);
\draw [thin, dashed] (0,2) to (0,1);
\draw [thin, dashed] (0,2) to (2,0);
\draw [thin, dashed] (-2,0) to (-1,0);
\draw [thin, dashed] (-2,0) to (0,2);
\draw [thin, dashed] (-2,0) to (0,-2);
\draw [thin, dashed] (-3,0) to (-2,0);
\draw [thin, dashed] (0,-3) to (0,-2);
\draw [thin, dashed] (0,-3) to (-3,0);
\draw [thin, dashed] (3,0) to (2,0);
\draw [thin, dashed] (3,0) to (0,-3);
\draw [thin, dashed] (0,2) to (0,3);
\draw [thin, dashed] (0,3) to (-3,0);
\draw [thin, dashed] (0,3) to (3,0);
\draw [black,fill] (0,3) circle [radius=0.1];
\draw [white,fill] (0,3) circle [radius=0.09];
\draw [black,fill] (3,0) circle [radius=0.1];
\draw [white,fill] (3,0) circle [radius=0.09];
\draw [black,fill] (0,-3) circle [radius=0.1];
\draw [white,fill] (0,-3) circle [radius=0.09];
\draw [black,fill] (-3,0) circle [radius=0.1];
\draw [white,fill] (-3,0) circle [radius=0.09];
\draw [black,fill] (-2,0) circle [radius=0.1];
\draw [white,fill] (-2,0) circle [radius=0.09];
\draw [black,fill] (0,2) circle [radius=0.1];
\draw [white,fill] (0,2) circle [radius=0.09];
\draw [black,fill] (2,0) circle [radius=0.1];
\draw [white,fill] (2,0) circle [radius=0.09];
\draw [black,fill] (0,-2) circle [radius=0.1];
\draw [white,fill] (0,-2) circle [radius=0.09];
\draw [black,fill] (0,-1) circle [radius=0.1];
\draw [white,fill] (0,-1) circle [radius=0.09];
\draw [black,fill] (-1,0) circle [radius=0.1];
\draw [white,fill] (-1,0) circle [radius=0.09];
\draw [black,fill] (1,0) circle [radius=0.1];
\draw [white,fill] (1,0) circle [radius=0.09];
\end{tikzpicture}
\end{tabular}
\end{adjustbox}
\caption{Brushing strategy for $P_{3}\square C_{4}$}
\label{fig:brushing} % ALWAYS PUT LABEL AFTER CAPTION!!
\end{figure}

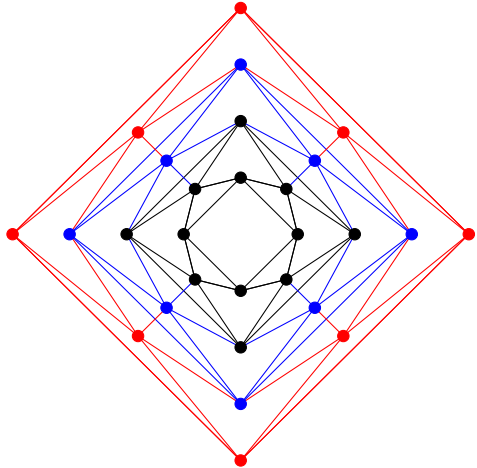
\begin{figure} 
  \begin{center}
    \begin{tikzpicture}[scale=.75]
\draw [black,thin] (-1,0) to (0,1);
\draw [black,thin] (-1,0) to (0,-1);
\draw [black,thin] (1,0) to (0,1);
\draw [black,thin] (1,0) to (0,-1);

\draw [black,thin] (-0.8,0.8) to (0,1);
\draw [black,thin] (-0.8,0.8) to (-1,0);
\draw [black,thin] (0.8,0.8) to (0,1);
\draw [black,thin] (0.8,0.8) to (1,0);
\draw [black,thin] (0.8,-0.8) to (0,-1);
\draw [black,thin] (0.8,-0.8) to (1,0);
\draw [black,thin] (-0.8,-0.8) to (0,-1);
\draw [black,thin] (-0.8,-0.8) to (-1,0);

\draw [black,thin] (2,0) to (0,2);
\draw [black,thin] (-2,0) to (0,-2);
\draw [black,thin] (-2,0) to (0,2);
\draw [black,thin] (0,-2) to (2,0);

\draw [black,thin] (-0.8,0.8) to (0,1);
\draw [black,thin] (-0.8,0.8) to (-1,0);
\draw [black,thin] (0.8,0.8) to (0,1);
\draw [black,thin] (0.8,0.8) to (1,0);
\draw [black,thin] (0.8,-0.8) to (0,-1);
\draw [black,thin] (0.8,-0.8) to (1,0);
\draw [black,thin] (-0.8,-0.8) to (0,-1);
\draw [black,thin] (-0.8,-0.8) to (-1,0);

\draw [black,thin] (2,0) to (0.8,0.8);
\draw [black,thin] (2,0) to (0.8,-0.8);

\draw [black,thin] (-2,0) to (-0.8,0.8);
\draw [black,thin] (-2,0) to (-0.8,-0.8);

\draw [black,thin] (0.8,0.8) to (0,2);
\draw [black,thin] (-0.8,0.8) to (0,2);

\draw [black,thin] (-0.8,-0.8) to (0,-2);
\draw [black,thin] (0.8,-0.8) to (0,-2);

\draw [blue,thin] (3,0) to (0,3);
\draw [blue,thin] (-3,0) to (0,-3);
\draw [blue,thin] (-3,0) to (0,3);
\draw [blue,thin] (0,-3) to (3,0);

\draw [blue,thin] (3,0) to (1.3,1.3);
\draw [blue,thin] (3,0) to (1.3,-1.3);

\draw [blue,thin] (-3,0) to (-1.3,1.3);
\draw [blue,thin] (-3,0) to (-1.3,-1.3);

\draw [blue,thin] (1.3,1.3) to (0,3);
\draw [blue,thin] (-1.3,1.3) to (0,3);

\draw [blue,thin] (-1.3,-1.3) to (0,-3);
\draw [blue,thin] (1.3,-1.3) to (0,-3);

\draw [blue,thin] (1.3,1.3) to (0.8,0.8);
\draw [blue,thin] (1.3,1.3) to (0,2);
\draw [blue,thin] (1.3,1.3) to (2,0);

\draw [blue,thin] (1.3,-1.3) to (0.8,-0.8);
\draw [blue,thin] (1.3,-1.3) to (0,-2);
\draw [blue,thin] (1.3,-1.3) to (2,0);

\draw [blue,thin] (-1.3,-1.3) to (-0.8,-0.8);
\draw [blue,thin] (-1.3,-1.3) to (0,-2);
\draw [blue,thin] (-1.3,-1.3) to (-2,0);

\draw [blue,thin] (-1.3,1.3) to (-0.8,0.8);
\draw [blue,thin] (-1.3,1.3) to (0,2);
\draw [blue,thin] (-1.3,1.3) to (-2,0);

\draw [red,thin] (4,0) to (0,4);
\draw [red,thin] (-4,0) to (0,-4);
\draw [red,thin] (-4,0) to (0,4);
\draw [red,thin] (0,-4) to (4,0);

\draw [red,thin] (4,0) to (0,4);
\draw [red,thin] (-4,0) to (0,-4);
\draw [red,thin] (-4,0) to (0,4);
\draw [red,thin] (0,-4) to (4,0);

\draw [red,thin] (4,0) to (1.8,1.8);
\draw [red,thin] (4,0) to (1.8,-1.8);

\draw [red,thin] (-4,0) to (-1.8,1.8);
\draw [red,thin] (-4,0) to (-1.8,-1.8);

\draw [red,thin] (1.8,1.8) to (0,4);
\draw [red,thin] (-1.8,1.8) to (0,4);

\draw [red,thin] (-1.8,-1.8) to (0,-4);
\draw [red,thin] (1.8,-1.8) to (0,-4);

\draw [red,thin] (1.8,1.8) to (1.3,1.3);
\draw [red,thin] (1.8,1.8) to (0,3);
\draw [red,thin] (1.8,1.8) to (3,0);

\draw [red,thin] (1.8,-1.8) to (1.3,-1.3);
\draw [red,thin] (1.8,-1.8) to (0,-3);
\draw [red,thin] (1.8,-1.8) to (3,0);

\draw [red,thin] (-1.8,-1.8) to (-1.3,-1.3);
\draw [red,thin] (-1.8,-1.8) to (0,-3);
\draw [red,thin] (-1.8,-1.8) to (-3,0);

\draw [red,thin] (-1.8,1.8) to (-1.3,1.3);
\draw [red,thin] (-1.8,1.8) to (0,3);
\draw [red,thin] (-1.8,1.8) to (-3,0);

\draw [black,fill] (-1,0) circle [radius=0.1];
\draw [black,fill] (1,0) circle [radius=0.1];
\draw [black,fill] (0,-1) circle [radius=0.1];
\draw [black,fill] (0,1) circle [radius=0.1];
\draw [black,fill] (-0.8,-0.8) circle [radius=0.1];
\draw [black,fill] (-0.8,0.8) circle [radius=0.1];
\draw [black,fill] (0.8,-0.8) circle [radius=0.1];
\draw [black,fill] (0.8,0.8) circle [radius=0.1];
\draw [black,fill] (-2,0) circle [radius=0.1];
\draw [black,fill] (2,0) circle [radius=0.1];
\draw [black,fill] (0,-2) circle [radius=0.1];
\draw [black,fill] (0,2) circle [radius=0.1];
\draw [blue,fill] (-3,0) circle [radius=0.1];
\draw [blue,fill] (3,0) circle [radius=0.1];
\draw [blue,fill] (0,-3) circle [radius=0.1];
\draw [blue,fill] (0,3) circle [radius=0.1];
\draw [blue,fill] (-1.3,-1.3) circle [radius=0.1];
\draw [blue,fill] (-1.3,1.3) circle [radius=0.1];
\draw [blue,fill] (1.3,-1.3) circle [radius=0.1];
\draw [blue,fill] (1.3,1.3) circle [radius=0.1];
\draw [red,fill] (-4,0) circle [radius=0.1];
\draw [red,fill] (4,0) circle [radius=0.1];
\draw [red,fill] (0,-4) circle [radius=0.1];
\draw [red,fill] (0,4) circle [radius=0.1];
\draw [red,fill] (-1.8,-1.8) circle [radius=0.1];
\draw [red,fill] (-1.8,1.8) circle [radius=0.1];
\draw [red,fill] (1.8,-1.8) circle [radius=0.1];
\draw [red,fill] (1.8,1.8) circle [radius=0.1];

       \end{tikzpicture}
  
  \end{center}
   \caption{The three layers of $L(P_{4} \square C_{4})$}
    
     \label{fig:layers} % ALWAYS PUT LABEL AFTER CAPTION!!
  
\end{figure} 

\iffalse
We leave it as an open problem to find a family of graphs (if one exists) to show that for no real number $c$ can it be that
$Z(L(G)) < cb(G)$ for all graphs $G$.
\fi

A similar result involving the more restricted brushing number $b(G)$ also holds:

\begin{theorem}\label{thm:NoReal-b}
There exists no real number $c$ for which $Z(L(G)) \leq c b(G)$ for all graphs $G$.
\end{theorem}

\begin{proof}
Consider the graph $\mathcal{G}_{k,6}$ that is obtained
by taking $k$ disjoint 6-cycles $\sC_1, \sC_2, \ldots, \sC_k$, each with vertices labelled 1 to 6 in cyclic order,
and then, for each $i \in \{1, 2, \ldots, k-1 \}$, identifying vertex 4 of cycle $\sC_i$ with vertex 1 of cycle $\sC_{i+1}$.
The result is a connected graph with $k-1$ cut vertices, and $b(\mathcal{G}_{k,6})=2$ but $Z(L(\mathcal{G}_{k,6})) = k+1$.
Clearly the ratio $\frac{Z(L(\mathcal{G}_{k,6}))}{b(\mathcal{G}_{k,6})}$ grows without bound as $k$ is allowed to increase.
\end{proof}

We previously noted that the difference $B(G)-Z(G)$ can be arbitrarily large (and either positive or negative),
depending on the choice of graph $G$.
We note that this behaviour also occurs for line graphs.
When considering the class of line graphs for $G = K_{1,n}$ we find that
$B(L(G))-Z(L(G))$ can be arbitrarily large and positive.
For examples of line graphs for which
$B(L(G))-Z(L(G))$ can be arbitrarily large but negative, consider the graph $\mathcal{G}_{k,6}$ introduced in the proof of Theorem~\ref{thm:NoReal-b}
and observe that $B(L(\mathcal{G}_{k,6}))=4$ but $Z(L(\mathcal{G}_{k,6})) = k+1$.

\section{Zero-forcing number of a graph vs.\ the zero-forcing number of its line graph}

It is known from~\cite{eroh} that $Z(G) \leq 2Z(L(G))$;
%, the proof of this is straightforward.
%% I want to avoid saying straightforward
each vertex in $L(G)$ that is in a zero-forcing
set corresponds to an edge of $G$, where
%It is not difficult to see that
the set of endpoints for all of these edges is a zero-forcing set for
$G$ of size at most $2Z(L(G))$. In this section we use a zero-forcing
set for $L(G)$ to construct a zero-forcing set for $G$ of the same size, thus proving a conjecture in~\cite{eroh}.

\begin{theorem} \label{theorem2}
If $G$ is a graph with no isolated vertices, then $Z(G)\leq Z(L(G))$.
\end{theorem}

\begin{proof}
  We will assume that $G$ is connected, as this implies the stated
  result.

  Let $Z$ be a zero-forcing set for $L(G)$.  Let $\sP_{1}, \ldots,
  \sP_{|Z|}$ be the %set of
  zero-forcing chains for a zero-forcing process
  starting with $Z$ in $L(G)$ (using the notation
  in~(\ref{notationpaths})).  We order this collection so that any
  paths with just one vertex are at the end of this collection.

  We will describe a strategy for choosing a zero-forcing set in $G$
  of size at most $|Z|$.  At each step in the zero-forcing process on
  $L(G)$ we describe which vertices are added to a set $S$ that will be a
  zero-forcing set for $G$. For each path in the zero-forcing chains
  at most one vertex in $G$ will be added to $S$.

  Suppose that at some point in the zero-forcing process on $L(G)$
  $w_{i,j}$ is the active vertex; so $w_{i,j}$ forces
  $w_{i,j+1}$. At this step $w_{i,j}$ in $L(G)$ is black, as are all
  neighbours of $w_{i,j}$, except for $w_{i,j+1}$ (some of these
  neighbours might be initially black vertices of $L(G)$, while others
  may have been forced at some earlier step). In $G$ consider the
  edges $w_{i,j} = \{a, b\}$ and $w_{i,j+1} = \{b,c\}$.

  If $w_{i,j}$ is an initially black vertex in $L(G)$, then include
  $a$ in $S$ and mark the chain $\sP_i$ as used.
  At this step, add all the white vertices in $(N_{G}(a) \cup
  N_{G}(b)) \setminus \{a,b,c\}$ to $S$: each such white vertex is an endpoint
  %these vertices are the endpoints
  of some edge in $G$ with the property that the vertex in $L(G)$ corresponding to this edge is from an unused path.
  %that corresponds to an initially black vertex
  %in $L(G)$ that are from unused paths.
  Mark each such path as
  used. %is this clear?

  Then $a$ forces $b$ in $G$ (if $b$ is not black already),
  since at this step in $G$ the only vertex adjacent to $a$ that could
  possibly be white is $b$. At this point $b$ can similarly force $c$
  (if $c$ is not black already).

  In both cases, $w_{i,j}$ zero forcing $w_{i,j+1}$ in $L(G)$
  corresponds to some zero-forcing steps in $G$ in which no white
  vertices are left incident with the edges $w_{i,j}$ and $w_{i,j+1}$.

  Continue like this for all steps of the zero-forcing process on
  $L(G)$. If there are any unused chains left, then each of these chains must consist of a single vertex, which is initially black. If $w =\{a,b\}$ is such a vertex of $L(G)$, then at this step in $L(G)$ all neighbours of $w$ are black; and any white vertex in $(N_{G}(a) \cup N_{G}(b)) \setminus \{a,b\}$ is an endpoint of some edge in $G$ such that the vertex in $L(G)$ corresponding to this edge is from an unused path of length $1$. Add all white vertices in $(N_{G}(a) \cup N_{G}(b)) \setminus \{a,b\}$ to $S$ (each such vertex corresponds to an endpoint of some unused zero-forcing path of $L(G)$).  %If $w =\{a,b\}$ is such a vertex of $L(G)$, then at this step in $G$ all vertices in $(N_{G}(a) \cup N_{G}(b)) \setminus \{a,b\}$ are black. 
If $a$ is also black, then it forces $b$ (if $b$ is not black already) and vice versa. If both $a$ and $b$ are white then include one of them, say $a$, in $S$; so $a$ forces $b$. At this point each vertex of $L(G)$ is black as is every
  vertex in $G$. %this is not clear
  Thus this procedure produces a zero-forcing set for $G$ of order at
  most $|Z|=Z(L(G))$.
\end{proof}

We also have the following result which is  parallel to Corollary~\ref{cor:isolateone}.
\begin{corollary}\label{cor:isolatetwo}
For any graph $G$ with $k$ isolated vertices, $Z(G) \leq Z(L(G)) + k$.
\end{corollary}

\section{Brushing number of a graph vs.\ the brushing number of its line graph}

Finally we prove that $B(G) \leq B(L(G))$ for any graph $G$ with no
isolated vertices.

\begin{theorem} \label{theorem3}
For any graph $G$ with no isolated vertices, $B(G)\leq B(L(G))$.
\end{theorem}

\begin{proof}
  As in the previous theorems, we can assume that $G$ is connected.
  Further, the theorem holds if $G$ is a single edge, so we can also
  assume that $L(G)$ is not a single vertex.

  Consider a brushing configuration $\mathcal{B}_{L(G)}$ of $L(G)$
  with $B(L(G))$ brushes. In this brushing configuration assume that
  the vertices fire in order $v_{1}, v_{2}, \ldots, \linebreak
  v_{|V(L(G))|}$. Using this ordering of $\mathcal{B}_{L(G)}$ we will
  choose an initial placement of at most $B(L(G))$ brushes at the
  vertices of $G$ to construct a brushing configuration
  $\mathcal{B}_{G}$ of $G$.

  We consider two types of vertices in $L(G)$. The first, called
  \textsl{type~1}, is the set of vertices that are not incident to any
  clean edges when they fire in $\mathcal{B}_{L(G)}$. The second
  set, called \textsl{type~2}, are the remaining vertices, so these
  vertices have at least one incident clean edge when they fire in
  $\mathcal{B}_{L(G)}$. Clearly the first vertex to fire, $v_1$, is a
  type~1 vertex.

  Assume that in the brushing process vertex $v$ fires.  If $v$ is
  type~1, then there must be at least $\deg_{L(G)}(v)$ brushes at $v$
  in $L(G)$ in the initial configuration of brushes in
  $\mathcal{B}_{L(G)}$ (since no edges incident with $v$ are clean, no
  new brushes have been sent to $v$).  Consider the edge $v=\{a,b\}$
  in $G$. Note that $\deg_{L(G)}(v) = \deg_G(a) + \deg_G(b) -2$ and we
  will assume that $\deg_{G}(a) \geq \deg_{G}(b)$.

  Since $G$ is not a single edge $\deg_{G}(a) \geq 2$. Put $\deg_{G}(a)-s-2$
  brushes at $a$ where $s$ is the current number of clean edges
  incident with $a$ in $G$. %does this have to be positive?
  Similarly, put $\deg_{L(G)}(v)-(\deg_{G}(a)-s-2)-t$ brushes at $b$
  where $t$ is the current number of clean edges incident with $b$ in
  $G$. Note that $\deg_{L(G)}(v)-(\deg_{G}(a)-s-2)-t$ is the current
  number of unclean edges at $b$ in $G$. So $b$ fires and cleans the
  edge $v$ (and possibly some other edges in $G$), and therefore a
  brush is sent from $b$ to $a$. This reduces the number of unclean
  edges incident with $a$ by 1 and increases the current number of
  brushes at $a$ from $\deg_{G}(a)-s-2$ to $\deg_{G}(a)-s-1$. So there
  are as many brushes at $a$ as the number of unclean edges incident
  with $a$, and hence $a$ fires in $G$.

  Suppose that $v = \{a,b\}$ is the first vertex in $v_{1}, v_{2}, \ldots,
  v_{|V(L(G))|}$ of type~2. Let $p$ be the number of clean edges
  incident with $v$ in $L(G)$ just before $v$ fires.  Then there must
  be at least $\deg_{L(G)}(v)-2p$ brushes at $v$ in $L(G)$ in the
  initial configuration of brushes in $\mathcal{B}_{L(G)}$ (the number
  of dirty edges incident with $v$ is $\deg_{L(G)}(v)-p$ and (at least) $p$
  brushes were sent to $v$ when the incident edges were cleaned).

  Our aim is to show that both vertices $a$ and $b$ in $G$ can fire
  (if they have not fired yet) after distributing at most
  $\deg_{L(G)}(v)-2p$ brushes among them. Any vertex $u$ that fires
  before $v$ in $\mathcal{B}_{L(G)}$ must be of type~1, and the brushing
  strategy for type~1 vertices guarantees that both endpoints
  of the edge $u$ in $G$ have already fired. So any clean edge
  $\{v,x\}$ in $L(G)$ has been cleaned because of the firing of the vertex $x$ in
  $L(G)$ at some earlier step, and so both endpoints of the edge $x$
  in $G$ must have already fired cleaning one of the endpoints, say
  $a$, of the edge $v=\{a,b\}$ in $G$ and the edge $v$ itself.  Put (at
  most) $\deg_{L(G)}(v)-2p$ brushes at $b$ in $G$, and so $b$ can fire.

  Consider the next vertex of type~2 in $\mathcal{B}_{L(G)}$. Note
  that each time a vertex $v$ in $L(G)$ of type~2 with an incident
  clean edge $\{v,w\}$ is considered in $\mathcal{B}_{L(G)}$, it must
  be that in $G$ both endpoints of the edge $w$ have already
  fired. Denote the edge $v$ in $G$ as $\{a, b\}$. Since $v \cap w
  \neq \emptyset$ in $G$, this means that either $a$ or $b$ has
  already fired and cleaned the edge $v$ in $G$. The degree arguments
  in the preceding paragraph can be repeated to show that the other
  endpoint of $v$ in $G$ also fires after putting the corresponding
  number of brushes there. This
  procedure
  can be repeated until all type~2
  vertices in $\mathcal{B}_{L(G)}$ have been considered.

  It is now easy to see that we have established a distribution of at most
  $B(L(G))$ brushes among the vertices of $G$ which cleans all edges
  and vertices of $G$.
\end{proof}

We state the following result which has essentially the same proof as
Corollary~\ref{cor:isolateone}.

\begin{corollary}\label{cor:isolatethree}
For any graph $G$ with $k$ isolated vertices, $B(G) \leq B(L(G)) + k$.
\end{corollary}

We remark that this method also works for the variant of brushing that
considers capacity constraints, as in the setting of~\cite{ MR2476825, MR3176702, MR2489423} and
as in the proof of Theorem~\ref{theorem3}. In particular, the construction in Theorem~\ref{theorem3} can be used to prove the following corollary.

%We remark that one can as well start with a brushing of $L(G)$
%considering capacity constraints as in the setting of~\cite{ MR2476825, MR3176702, MR2489423} and
%as in the proof of Theorem~\ref{theorem3} construct a brushing of $G$ %in the same setting 
%with at most
%$b(L(G))$ brushes, leading to the following corollary.

\begin{corollary} \label{corollary2}
For any graph $G$ with no isolated vertices, $b(G)\leq b(L(G))$.
\end{corollary}

Similar to Corollary \ref{cor:isolatethree}, we obtain the following result.

\begin{corollary}\label{cor:isolatefour}
For any graph $G$ with $k$ isolated vertices, $b(G) \leq b(L(G)) + k$.
\end{corollary}

\section{Further Work}

\iffalse
We previously noted that the difference $B(G)-Z(G)$ can be arbitrarily large (and either positive or negative),
depending on the choice of graph $G$.
We observe that this behaviour also occurs for line graphs.
When considering the class of line graphs for $G = K_{1,n}$ we find that
$B(L(G))-Z(L(G))$ can be arbitrarily large and positive.
For examples of line graphs for which
$B(L(G))-Z(L(G))$ can be arbitrarily large but negative, consider the graph $\mathcal{G}_{k,6}$ that is obtained
by taking $k$ disjoint 6-cycles $\sC_1, \sC_2, \ldots, \sC_k$, each with vertices labelled 1 to 6 in cyclic order,
and then, for each $i \in \{1, 2, \ldots, k-1 \}$, identifying vertex 4 of cycle $\sC_i$ with vertex 1 of cycle $\sC_{i+1}$.
The result is a connected graph with $k-1$ cut vertices, and $B(L(\mathcal{G}_{k,6}))=4$ but $Z(L(\mathcal{G}_{k,6})) = k+1$.

%Other graphs for which $B(L(G)) \geq Z(L(G))$ are also easily found,
%so we instead ask:  for which graphs $G$ is it true that $B(L(G)) < Z(L(G))$?

In Theorem~\ref{thm:NoReal} we exhibited a family of graphs that allowed us to conclude that
there is no real number $c$ for which $Z(L(G)) \leq cB(G)$ for every graph $G$.
We leave it as an open problem to find a family of graphs (if one exists) to show that for no real number $c$ can it be that
$Z(L(G)) \leq cb(G)$ for all graphs $G$.
\fi

The inequalities of Theorems~\ref{theorem1}, \ref{theorem2} and~\ref{theorem3}
are all tight when $G = C_n$ with $n \geq 3$ and when $G=P_{n}$ with $n \geq 2$.
We also note that the graphs $\mathcal{G}_{k,6}$ from Theorem~\ref{thm:NoReal-b}
yield a family of examples with equality for Theorem~\ref{theorem2} (as does any natural generalization with arbitrary cycle lengths). It would be interesting to know what other
classes of graphs cause these bounds to hold with equality.
%I think equality holds in Theorem~\ref{theorem1} for the graphs $P_2 \Box P_3$ and $P_3 \Box P_3$.

An algebraic property of the zero-forcing number is that it is bounded below by the maximum nullity of the graph (as defined in ~\cite{MR2388646}). Indeed, it was this property that initially motivated the study of zero forcing. It would be interesting to try to connect the maximum nullity of $L(G)$, or some other algebraic property of $L(G)$, to $B(G)$.

% An algebraic property of the zero-forcing number is that it is bounded below by the maximum nullity of the graph (maximum nullity is defined in~\cite{MR2388646}). Is there any way to connect the maximum nullity of $L(G)$ to $B(G)$?

\section{Acknowledgements}

A.~Erzurumluo\u{g}lu acknowledges research support from AARMS.
K.~Meagher acknowledges research support from NSERC (grant number RGPIN-341214-2013).
D.A.~Pike acknowledges research support from NSERC (grant number RGPIN-04456-2016).

%%%%%%%%%%%%%

\end{document}